\documentclass{article}
\usepackage{amssymb, amsthm, amsmath, amsfonts}
\usepackage{geometry}
\usepackage[colorlinks=true,urlcolor=blue,linkcolor=red,citecolor=magenta]{hyperref}

\usepackage[backend=biber,maxbibnames=9, hyperref,
            style=numeric]{biblatex}
\providecommand{\keywords}[1]{\textbf{Keywords} #1 \\[12px]}
\providecommand{\subclass}[1]{\textbf{Mathematics Subject Classification (2020)} #1 }

\newtheorem{theorem}{Theorem}[section]
\newtheorem{lemma}[theorem]{Lemma}
\newtheorem{corollary}[theorem]{Corollary}
\newtheorem{proposition}[theorem]{Proposition}

\theoremstyle{definition}
\newtheorem{definition}[theorem]{Definition}
\newtheorem{remark}[theorem]{Remark}

\newtheorem{openp}[theorem]{Open Problem}

\DeclareMathOperator{\wt}{wt} 
\DeclareMathOperator{\im}{Im}
\DeclareMathOperator{\mult}{mult}
\DeclareMathOperator{\Tr}{Tr}
\DeclareMathOperator{\NL}{\mathcal{NL}}
\newcommand{\graph}[1]{\mathcal{G}_{#1}}
\newcommand{\Z}{\mathbb{Z}}
\newcommand{\F}{\mathbb{F}}
\newcommand{\set}[1]{\left\{ #1 \right\}}
\newcommand{\parens}[1]{ \left( #1 \right)}
\newcommand{\ceil}[1]{ \left\lceil #1 \right\rceil }
\newcommand{\bentcomps}[1]{{\mathcal{B}({#1})}}

\begin{document}

\title{On lower bounds for the distances between APN functions}

\author{Maria Mihaila$^1$, Darrion Thornburgh$^{1}$\thanks{Corresponding Author}}
\date{$^1$Vanderbilt University 
\\ \texttt{maria.mihaila@vanderbilt.edu, darrion.thornburgh@vanderbilt.edu}\\[2ex]%
}

\maketitle
\begin{abstract}
Whether two distinct APN functions can have a Hamming distance of $1$ remains an open problem. In 2020, L. Budaghyan et al. introduced a new CCZ-invariant $\Pi_F$ which can be used to provide lower bounds on the Hamming distance between a given APN function $F \colon \mathbb{F}_2^n \to \mathbb{F}_2^n$ and other APN functions. Lower bounds on the distance from an APN function $F$ to any other APN function $G$ are known when $F$ is an almost bent (AB) function or when $F$ is a $3$-to-$1$ quadratic function with $n$ even. 
In this paper, we reinterpret $\Pi_F$ in terms of the multiplicities of the 3-sums of the graph $\mathcal{G}_F=\{(x, F(x)) : x \in \mathbb{F}_2^n\}$ of $F$ as a Sidon set, which we call exclude multiplicities.
For even $n$, we establish lower bounds on the distance between $F$ and any other APN function $G$ when $F$ is plateaued APN, and we generalize a previously known lower bound for quadratic $3$-to-$1$ functions to the case where $F$ is plateaued $3$-to-$1$ (e.g., when $F$ is a Kasami function).
For odd $n$, we derive new lower bounds when $F$ is the APN inverse function over $\mathbb{F}_{2^n}$.
We also study how the exclude multiplicities of $\mathcal{G}_F$ are directly connected to the existence of linear structures of $\gamma_F$ when $F$ is plateaued APN and to the ortho-derivative when $F$ is a quadratic APN function. 
In particular, we prove that $\gamma_F$ has no nontrivial linear structures when $F$ is plateaued APN.
We also use the CCZ-invariance of exclude multiplicities to prove that the Brinkmann-Leander-Edel-Pott function is not CCZ-equivalent to a plateaued function.
\end{abstract}


\subclass{Primary: 94D10, 11B13. \\ Secondary: 94A60.}

\section{Introduction}

We call a function $F \colon \F_2^n \to \F_2^m$ a \textit{vectorial Boolean function} or an \textit{$(n,m)$-function}, where $\F_2^n$ is an $n$-dimensional vector space over $\F_2$.
Vectorial functions are prevalent and are a major focus of study in cryptography, and in particular, they are used as the primary nonlinear components of block ciphers. 
In this paper, we are concerned with the $(n,n)$-functions that are optimally resistant to a differential attack, the attack on a block cipher first introduced by Biham and Shamir in \cite{BihamShamir1991}.
Such functions are called \textit{almost perfect nonlinear} (APN).
In particular, we call $F$ APN if $F(x)+F(x+a)=b$ has either $0$ or $2$ solutions for all $a,b \in \F_2^n$ where $a \neq 0$.

An active area of research is the construction of new APN functions.
In particular, it is conjectured that changing the value of an APN function at a single point results in a function that is no longer APN \cite{budaghyanCarletHellesetUpperBoundsDegree}.
For two vectorial Boolean functions $F,G \colon \F_2^n \to \F_2^n$, we define the \textit{Hamming distance} between $F$ and $G$ as 
\[
d(F,G) = |\set{x \in \F_2^n : F(x) \neq G(x)}|.
\]
Said differently, it is conjectured that any two distinct APN functions have Hamming distance greater than $1$ (for $n \geq 3$), and it is known that this conjecture holds in certain cases, such as when $F$ is a plateaued APN function and $G$ is an arbitrary APN function.
Moreover, if two APN functions have a distance of $1$, at least one of them has algebraic degree $n$ (see \cite{carlet_apnGraphMaximal}), and it is also conjectured that no APN function has algebraic degree $n$ \cite{budaghyanCarletHellesetUpperBoundsDegree}.

In \cite{DistanceBetweenAPNFunctions2020}, the Carlet-Charpin-Zinoviev (CCZ) equivalence invariant $\Pi_F$ was first introduced (see Section~\ref{sec:distance-between-apn-functions} for the definition of $\Pi_F$).
It was shown that $\Pi_F$ can be used to determine a lower bound on the Hamming distance between distinct APN functions $F$ and $G$.
From this, the authors of \cite{DistanceBetweenAPNFunctions2020} derived lower bounds on the distance between any Gold function $F(x) = x^3$ and any other APN function, and they observed that their lower bounds all tend to infinity as $n$ grows \cite{DistanceBetweenAPNFunctions2020}.
The lower bounds on the Hamming distance between the APN function $F(x)=x^3$ and any other APN function found in \cite{DistanceBetweenAPNFunctions2020} were generalized to the cases where $F$ is an almost bent (AB) function with $n$ odd \cite{CoulterKaleyski2021}, and where $F$ is a quadratic $3$-to-$1$ function with $n$ even \cite{BudaghyanTriplicateFunctions}.
Besides AB functions and quadratic $3$-to-$1$ functions, no other strong lower bounds on distances between APN functions were known until this paper.

In this paper, we have three main results. 
The first is a new lower bound on the distance between two APN functions $F$ and $G$ when at least one of them is plateaued, and our new bound significantly improves on the previous lower bound of $d(F,G) \geq 2$.
We then generalize the result of \cite{BudaghyanTriplicateFunctions} from quadratic $3$-to-$1$ functions to all plateaued $3$-to-$1$ functions (which are necessarily APN).
Note that these first two results are stated for all even $n \geq 4$ because any plateaued APN function is AB when $n$ is odd, and we address this case later in the paper.
The third main result concerns the APN inverse function and provides the first known lower bounds in terms of binary Kloosterman sums $K_n(a) = \sum_{x \in \F_{2^n}}(-1)^{\Tr(ax + x^{-1})}$.

\begin{theorem}\label{thm:plat-dist-bound}
    Suppose $n\geq 4$ is even, let $F,G \colon \F_2^n \to \F_2^n$ be distinct APN functions such that $F$ is plateaued.
    Then $d(F,G) \geq 2^{\frac{n}{2}-1}$.
\end{theorem}

\begin{theorem}\label{thm:distance-to-plat3to1}
    Suppose $n \geq 4$ is even, let $F,G \colon \F_2^n \to \F_2^n$ be distinct APN functions such that $F$ is plateaued and $3$-to-$1$.
    Then 
    \[
        d(F,G) \geq 
        \begin{cases}
            \frac{1}{3}(2^{n-1} - 2^{\frac{n}{2}}+2) & n \equiv 0 \mod 4, \\
            \frac{1}{3}(2^{n-1} - 2^{\frac{n}{2}-1}+2) & n \equiv 2 \mod 4.
        \end{cases}
    \]
\end{theorem}

\begin{theorem}\label{thm:dist-to-Inverse}
    Suppose $n\geq 3$ is odd.
    Let $F \colon \F_{2^n} \to \F_{2^n}$ be the function $F(x) =x^{-1}$, where we define $\frac{1}{0}:=0$.
    If $G \colon \F_{2^n} \to \F_{2^n}$ is an APN function such that $G \neq F$, then
    \begin{align*}
    d(F,G) \geq \ceil{\frac{2^n-5 - \max_{v \in \F_{2^n}^\ast} |K_n(v)-1|}{6}} +1 \approx \frac{2^n-2^{\frac{n}{2}+1}}{6}+1.
    \end{align*}
\end{theorem}

Theorems \ref{thm:plat-dist-bound}, \ref{thm:distance-to-plat3to1}, and \ref{thm:dist-to-Inverse} are proven as immediate consequences of Theorems \ref{thm:plat-exclude-bound}, \ref{thm:plateaued-3to1-mults}, and \ref{thm:inverse-minmult}, respectively.

A \textit{Sidon set} in $\F_2^n$ is a subset $S \subseteq \F_2^n$ such that no four distinct points in $S$ have zero sum. 
APN functions and Sidon sets are closely related as the graph $\graph F = \set{(x, F(x)) : x \in \F_2^n}$ of a function $F \colon \F_2^n \to \F_2^n$ is a Sidon set in $(\F_2^n)^2$ if and only if $F$ is APN (see, for instance, \cite[Section 11.3.2]{CarletBook}).
We recall the following definitions, which were first introduced in \cite{quadspaper}.
\begin{definition}
For a Sidon set $S \subseteq \F_2^n$ and a point $p \in \F_2^n \setminus S$, we say that the \textit{exclude multiplicity} $\mult_S(p)$ of $p$ is 
\[
\mult_S(p) = |\set{\set{x,y,z} \subseteq S : x+y+z =p}|.
\]
If $\mult_S(p) >0$, we call $p$ an \textit{exclude point} of $S$. 
\end{definition}

Indeed, for a Sidon set $S \subseteq \F_2^n$, the exclude multiplicity $\mult_S(p)$ of some point $p \in \F_2^n \setminus S$ is equal to the number of 2-dimensional affine subspaces in $S \cup \set{p}$, which are necessarily not contained in a Sidon set by definition.
We call a Sidon set $S$ \textit{maximal} if it cannot be contained in a strictly larger Sidon set. Equivalently, any point in $\F_2^n \setminus S$ has nonzero exclude multiplicity with respect to $S$.
Carlet demonstrated in \cite{carlet_apnGraphMaximal} that the graph of any APN function is a maximal Sidon set if and only if any two distinct APN functions have a Hamming distance strictly greater than $1$.

All of our main theorems rely on lower bounds for the minimum exclude multiplicity of a point with respect to the graph of an APN function $F$, namely $e_{\min}(\graph F) = \min_{(a,b) \in (\F_2^n)^2 \setminus \graph F} \mult_{\graph F}(a,b)$. 
This is because we will derive the inequality
\[
d(F,G) \geq e_{\min}(\graph F)+1,
\]
which holds for any APN function $G \neq F$.
Moreover, we will show that this lower bound is equivalent to the one obtained in \cite{DistanceBetweenAPNFunctions2020} via the CCZ-invariant $\Pi_F$.
The known lower bounds on the distances for different families of APN functions are summarized in Table~\ref{tab:main-results}.

\renewcommand{\arraystretch}{2}
\begin{table}[ht!]
    \centering
    \begin{tabular}{p{3cm}|c|c}
    $F \colon \F_{2^n} \to \F_{2^n} $ & Lower bounds on $d(F,G)$ for $G \neq F$ and $G$ APN & Reference \\
    \hline
        $F$ almost bent (AB) & $\frac{1}{3}(2^{n-1}+2)$ & \cite{CoulterKaleyski2021},   
        Theorem~\ref{thm:dist-to-AB-functions}  \\
        $F$ plateaued APN, $n\geq 4$ even & $2^{\frac{n}{2}-1}$ & Theorem~\ref{thm:plat-dist-bound} \\
         $F$ plateaued $3$-to-$1$, $n\geq 4$ even &
         $
        \begin{cases}
            \frac{1}{3}(2^{n-1} - 2^{\frac{n}{2}}+2) & n \equiv 0 \mod 4, \\
            \frac{1}{3}(2^{n-1} - 2^{\frac{n}{2}-1}+2) & n \equiv 2 \mod 4.
        \end{cases}
        $
        & Theorem~\ref{thm:distance-to-plat3to1}  \\
        $F(x)=x^{-1}$, $n\geq3$ odd & 
        $\ceil{\frac{1}{6} \parens{2^n-5 - \max_{a \in \F_{2^n}^\ast} |K_n(a)-1|}} +1$
        & Theorem~\ref{thm:dist-to-Inverse} 
    \end{tabular}
    \caption{For an APN function $F$ in a given family, its Hamming distance to any other APN function is greater than or equal to the value in the second column.}
    \label{tab:main-results}
\end{table}
\renewcommand{\arraystretch}{1}

The remainder of this paper is organized as follows. 
In Section~\ref{sec:prelim}, we provide the necessary background for this paper. 
In Section~\ref{sec:distance-between-apn-functions}, we recall results of \cite{DistanceBetweenAPNFunctions2020} and provide a new description of $\Pi_F$ which is equivalent to its original definition.

In Section~\ref{sec:plateaued}, we study the exclude multiplicities of the graphs of plateaued APN functions, and we prove Theorems~\ref{thm:plat-dist-bound} and \ref{thm:distance-to-plat3to1}. 
We express the exclude multiplicities of $\graph F$ in terms of a sum involving the amplitudes of the components of $F$ (see the preliminaries).
Moreover, we generalize \cite[Proposition 5]{DistanceBetweenAPNFunctions2020} to all plateaued APN functions, showing that $\Pi_F$ is completely determined by $\Pi_F^0$ (see Section~\ref{sec:distance-between-apn-functions}).
As a corollary, we prove that the Brinkmann-Leander-Edel-Pott function (the only known instance of an APN function not CCZ-equivalent to a power function or a quadratic function) is also not CCZ-equivalent to a plateaued function.
We also prove that the exclude multiplicities of $\graph F$ are odd when $F$ is plateaued APN ($n \geq 3$) via a characterization in terms of the intersections of $\bentcomps{F}=\set{v \in \F_2^n : v \cdot F \text{ is bent}}$ with affine hyperplanes.
In Section~\ref{subsec:quadratic-functions}, we study quadratic APN functions with a focus on how their ortho-derivatives can be used to express exclude multiplicities.

In Section~\ref{sec:power-fcns}, we study the exclude multiplicities of points with respect to the graphs of APN power functions.
As a consequence of a result in Section~\ref{sec:power-fcns}, we resolve, in the case of $n$ odd, a conjecture posed by Kaleyski in 2019.
Then, in Section~\ref{sec:inverse}, we prove Theorem~\ref{thm:dist-to-Inverse}.
We use results of Lachaud and Wolfmann in \cite{LachaudWolfmannGoppa} on elliptic curves over $\F_{2^n}$ and their connection to binary Kloosterman sums in order to derive our lower bounds in the case of $F(x)=x^{-1}$ when $n$ is odd.

In Section~\ref{sec:gammaF-linearstr}, we study how the exclude multiplicities of $\graph F$ are related to the $\gamma_F$ Boolean function, where $F$ is an APN function (see Section~\ref{sec:prelim} for the definition of $\gamma_F$).
In 2021, Carlet asked whether $\gamma_F$ can have nontrivial linear structures when $F$ is an arbitrary APN function \cite{GammaCarlet}.  
We resolve this problem in the case where $F$ is plateaued APN by proving that this is impossible.
To conclude the paper, we list open problems in Section~\ref{sec:open-ques}.

\section{Preliminaries}\label{sec:prelim}

For a Boolean function $f\colon \F_2^n \to \F_2$, we say that its (Hamming) \textit{weight} is the size of its support, i.e. $\wt(f) = \sum_{x \in \F_2^n} f(x)$.
We say that $f$ is \textit{balanced} if $\wt(f) = 2^{n-1}$.
We also define the \textit{linearity} $\mathcal{L}(f)$ of $f$ as the maximum absolute value of its \textit{Walsh transform} $W_f \colon \F_2^n \to \Z$ defined by 
\[
W_f(u) = \sum_{x \in \F_2^n} (-1)^{f(x) + x \cdot u}
\]
for all $u \in \F_2^n$.
Hence, $\mathcal{L}(f) = \max_{u \in \F_2^n}|W_f(u)|$, and moreover, the \textit{nonlinearity} of $f$ is $\NL(f) = 2^{n-1} -\frac{1}{2}\mathcal{L}(f)$.
A function is \textit{bent} if and only if its linearity is minimal with $\mathcal{L}(f) = 2^{\frac{n}{2}}$, and so bent functions only exist when $n$ is even.
Equivalently, $f$ is bent if and only if $D_a f(x) = f(x) + f(x+a)$ is balanced for all nonzero $a \in \F_2^n$.
Also, a Boolean function $f$ is \textit{quadratic} if $D_af$ is an affine function for all $a \in \F_2^n$.
We say that $f$ is \textit{plateaued} if there exists a non-negative integer $\lambda$ such that $W_f(u) \in \set{0, \pm \lambda}$ for all $u \in \F_2^n$, and we call $\lambda$ the \textit{amplitude} of $f$.
It is well-known that all quadratic and bent Boolean functions are plateaued.

For an $(n,m)$-function $F$, we call $D_a F(x) = F(x)+F(x+a)$ the \textit{derivative of $F$ in the direction of $a\in\F_2^n$}.
By definition, APN functions are those $(n,n)$-functions $F$ such that $D_a F$ has image size $2^{n-1}$ for all nonzero $a \in \F_2^n$.
For a vectorial Boolean function $F \colon \F_2^n \to \F_2^n$, the \textit{Walsh transform} of $F$ is defined as the function $W_F \colon (\F_2^n)^2 \to \Z$ defined by 
\[
W_F(u,v)= \sum_{x \in \F_2^n} (-1)^{u \cdot x+ v \cdot F(x)}
\]
for all $(u,v)\in(\F_2^n)^2$, where ``$\cdot$'' is the standard inner product on $\F_2^n$.
Note that $W_F(u,v)$ is equal to the Walsh transform of the \textit{component function} $(v \cdot F)(x) = v \cdot F(x)$ of $F$ evaluated at $u$.
We denote the set of all bent component functions of $F$ as $\bentcomps{F}$.

It is a classical result of \cite{chabaud_vaudenay_1995} that $F$ is APN if and only if $\sum_{(u,v) \in (\F_2^n)^2}W_F^4(u,v) = 2^{2n}(3 \cdot 2^{2n} -2^{n+1})$.
We say that the \textit{linearity} of $F$ is the value $\mathcal{L}(F) = \max_{(u,v) \in (\F_2^n)^2 \setminus (0,0)} |W_F(u,v)|$, and the \textit{nonlinearity} of $F$ is $\NL(F) = 2^{n-1} - \frac{1}{2}\mathcal{L}(F)$.
Oftentimes, we identify $\F_2^n$ with $\F_{2^n}$, in which case, we define $u \cdot v$ as $\Tr(uv)$ where $\Tr$ is the \textit{absolute trace function} $\Tr(x)=\sum_{i=0}^{n-1}x^{2^i}$ from $\F_{2^n}$ onto $\F_2$.
The \textit{Fourier-Hadamard transform} of a function $\varphi \colon \F_2^n \to \Z$ is the function $\widehat{\varphi} \colon \F_2^n \to \Z$ defined by 
\[
\widehat{\varphi}(u) = \sum_{x \in \F_2^n}(-1)^{x \cdot u} \varphi(x)
\]
for all $u \in \F_2^n$.
The Walsh transform of $F$ can also be expressed as the Fourier-Hadamard transform of the indicator $1_{\graph F}$ of $\graph F$ because we have 
\[
W_F(u,v) = \sum_{x \in \F_2^n }(-1)^{u \cdot x + v \cdot F(x)} = \sum_{(x,y) \in (\F_2^n)^2}(-1)^{u \cdot x + v \cdot y} 1_{\graph F}(x,y)=\widehat{1_{\graph F}}(u,v).
\]
For any integer $k \geq 1$, we have the useful equality 
\begin{equation}\label{eq:WFk-set-size}
    \widehat{W_F^k}(a,b) = 2^{2n}\left|\set{(x_1, \dots, x_k) \in (\F_2^n)^k : \sum_{i=1}^k (x_i, F(x_i)) = (a,b)}\right|,
\end{equation}
which follows directly from the definition of the Fourier-Hadamard transform.

For any two functions $\varphi,\psi \colon \F_2^n \to \Z$, define the \textit{convolutional product} of $\varphi$ and $\psi$ to be the function $\varphi \otimes \psi \colon \F_2^n \to \Z$ such that $(\varphi \otimes \psi)(u)=\sum_{x \in \F_2^n}\varphi(x) \psi(x+u)$ for all $u \in \F_2^n$.
The \textit{convolution law} is the equality $\widehat{\varphi \otimes \psi} = \widehat{\varphi} \widehat{\psi}$ for any functions $\varphi, \psi \colon \F_2^n \to \Z$ (see \cite[Proposition 11]{CarletBook}).

Another important class of vectorial Boolean functions are those such that $W_F(u,v) \in \set{0, \pm 2^{\frac{n+1}{2}}}$ for all $(u,v) \in (\F_2^n)^2 \setminus \set{(0,0)}$.
If $F \colon \F_2^n \to \F_2^n$ has this property, then we call $F$ \textit{almost bent} (AB).
Note that AB functions only exist when $n$ is odd as the Walsh transform always takes integer values, and any AB function is APN (see \cite{chabaud_vaudenay_1995} or \cite{vandamflass}, for instance).

We call two vectorial Boolean functions $F, G \colon \F_2^n \to \F_2^n$ \textit{CCZ-equivalent} if there exists an affine permutation $A \colon (\F_2^n)^2 \to (\F_2^n)^2$ such that $A(\graph F) = \graph G$.
It is conjectured that all APN functions of the form $F(x)=x^d$, defined over $\F_{2^n}$, are known up to CCZ-equivalence \cite{Dobbertin1999Niho}.
In Table~\ref{tab:APN-inf-families} and Table~\ref{tab:ABinf-families}, we list all known families of APN and AB monomials over $\F_{2^n}$.

\begin{table}[ht!]
    \centering
    \begin{tabular}{c|c|c|c}
        Name & $d$ & Condition & Reference(s) \\
        \hline
        Gold & $2^k + 1$ & $\gcd(k,n) = 1$ & 
        \cite{GoldR1968,NybergBook1994} 
        \\
        Kasami & $2^{2k} - 2^k + 1$ & $\gcd(k,n) = 1$  & 
        \cite{JanwaWilsonKasami1993,Kasami1971}
        \\
        Welch & $2^t + 3$ & $n = 2t + 1$ &   
        \cite{Dobbertin1999Welch} 
        \\ 
        Niho & 
        $\begin{cases}
            2^t + 2^{\frac{t}{2}} -1 & \text{if $t$ even} \\
            2^t + 2^{\frac{3t+1}{2}} -1 & \text{if $t$ odd}
        \end{cases}$
        & $n = 2t+1$ & 
        \cite{Dobbertin1999Niho} 
        \\ 
        Inverse & $2^{2t} -1$ & $n = 2t+1$ &   
        \cite{NybergBook1994,BethDing1994Inverse}
        \\
        Dobbertin & $2^{4t} + 2^{3t} + 2^{2t} + 2^t - 1$ & $n = 5t$ &   
        \cite{Dobbertin2001Dobbertin}
        \\
    \end{tabular}
    \caption{Known infinite families of APN power functions $\F_{2^n} \to \F_{2^n}$ of the form $x \mapsto x^d$.}
    \label{tab:APN-inf-families}
\end{table}

\begin{table}[ht!]
    \centering
    \begin{tabular}{c|c|c|c}
        Name & $d$ & Condition & Reference(s) \\
        \hline
        Gold & $2^k + 1$ & $\gcd(k,n) = 1$ & 
        \cite{GoldR1968,NybergBook1994}
        \\
        Kasami & $2^{2k} - 2^k + 1$ & $\gcd(k,n) = 1$  & 
        \cite{Kasami1971}
        \\
        Welch & $2^t + 3$ & $n = 2t + 1$ &   
        \cite{WeightDivisCCD2001,BinarymSequencesCCD2001}
        \\
        Niho & 
        $\begin{cases}
            2^t + 2^{\frac{t}{2}} -1 & \text{if $t$ even} \\
            2^t + 2^{\frac{3t+1}{2}} -1 & \text{if $t$ odd}
        \end{cases}$
        & $n = 2t+1$ & 
        \cite{HollmannXiang2001Niho}
    \end{tabular}
    \caption{Known infinite families of AB power functions $\F_{2^n} \to \F_{2^n}$ of the form $x \mapsto x^d$, $n$ odd.}
    \label{tab:ABinf-families}
\end{table}

Also of importance are plateaued and quadratic vectorial functions, which are not necessarily APN.
We call a function $F \colon \F_2^n \to \F_2^n$ \textit{plateaued} (resp., \textit{quadratic}) if all of its component functions $v \cdot F$ are plateaued (resp., quadratic).
Equivalently, $F$ is plateaued if for all $v \in \F_2^n$, there exists an integer $\lambda_v \geq 0$ such that $W_F(u,v) \in \set{0, \pm \lambda_v}$ for all $u \in \F_2^n$.
For example, any AB function is plateaued by definition.
Almost all known families of APN functions only contain functions that are CCZ-equivalent to monomials or quadratic functions. 

For a set $S \subseteq \F_2^n$, define $\delta_S \colon \F_2^n \to \Z_{\geq 0}$ to be the function given by $\delta_S(a) = |S \cap (a+S)|$ for all $a \in \F_2^n$.
Also, define the Boolean function $\gamma_S \colon \F_2^n \to \F_2$ such that $\gamma_S(a) = 1$ if $a \neq 0$ and $\delta_S(a)>0$ and $\gamma_S(a)=0$ otherwise.
When $n$ is even and $S = \graph F$ for some vectorial Boolean function $F \colon \F_2^{n/2} \to \F_2^{n/2}$, we simply denote $\delta_S$ as $\delta_F$ and $\gamma_S$ as $\gamma_F$.
The functions $\gamma_F$ and $\delta_F$ associated to a vectorial function $F$ were first introduced in \cite{carletCharpinZinovievCodesBentDES}.

Note that a set $S \subseteq \F_2^n$ is Sidon if and only if $\delta_S(a) \leq 2$ for all $a \in \F_2^n\setminus \set{0}$. 
Since $\delta_S$ cannot take value $1$, we have $\delta_S(a) \in\set{0,2}$ for all $a \in \F_2^n \setminus \set{0}$ precisely when $S$ is Sidon.
Therefore, we have that $\delta_S = 2\gamma_S + |S|\delta_0$ if and only if $S$ is Sidon, where $\delta_0$ denotes the indicator function of $\set{0}$ ($\delta_0$ is called the \textit{Dirac delta function} and it is not to be confused with $\delta_S$ or $\delta_F$).
In the particular case that we have a vectorial function $F \colon \F_2^n \to \F_2^n$, the value that $\delta_F(a,b)$ takes is equal to the number of solutions $F(x)+F(x+a)=b$.
A classical fact proved in \cite{chabaud_vaudenay_1995} is that $\delta_F = 1_{\graph F} \otimes 1_{\graph F}$ (indeed, this holds more generally with $\delta_S = 1_S \otimes 1_S$).

\section{Distance between APN functions}\label{sec:distance-between-apn-functions}

In this section, we establish the following lemma in two ways: first we prove it using the notion of exclude multiplicities, and then we demonstrate that the notation and methods of \cite{DistanceBetweenAPNFunctions2020} can also be used to derive the same bound.

\begin{lemma}\label{lem:MinimumDistance-By-MinimumExcludeMult}
    Let $F,G \colon \F_2^n \to \F_2^n$ be APN functions such that $F \neq G$.
    Then 
    \[
    d(F, G) \geq e_{\min}(\graph F) + 1,
    \]
    where $e_{\min}(\graph F)$ denotes $\min_{(a,b) \in (\F_2^n)^2 \setminus \graph F}\mult_{\graph F}(a,b)$.
\end{lemma}
\begin{proof}
    Suppose $F, G \colon \F_2^n\to \F_2^n$ are distinct APN functions.
    Let $x \in \F_2^n$ such that $F(x) \neq G(x)$, and let $k=\mult_{\graph F}(x, G(x))$.
    If $k =0$, then $d(F,G) \geq 1 = e_{\min}(\graph F)+1$.
    So, assume $k > 0$.
    By definition, there exist distinct (and necessarily pairwise disjoint) triples $\set{(x_i, F(x_i)), (y_i, F(y_i)), (z_i, F(z_i))}_{i=1}^k$ such that $(x_i+y_i+z_i,F(x_i)+F(y_i)+F(z_i))=(x,G(x))$ for all $1 \leq i \leq k$.
    For any $1 \leq i \leq k$, the equation $x_i + y_i + z_i = x$ implies that $x \notin \set{x_i, y_i, z_i}$ as $|\set{x_i, y_i,z_i}|=3$.
    Also, $\graph G$ is Sidon, so $\set{(x_i, F(x_i)), (y_i, F(y_i)), (z_i, F(z_i))}$ is not contained in $\graph G$ for any $1 \leq i \leq k$.
    Hence, there exist $a_1, \dots, a_k$ such that $a_i \in \set{x_i, y_i, z_i}$ for all $1 \leq i \leq k$ and $(a_i, F(a_i)) \notin \graph G$.
    Therefore, $F$ and $G$ differ at $x, a_1, \dots, a_k$, implying $d(F,G) \geq k+1 \geq e_{\min}(\graph F)+1$.
\end{proof}

\begin{remark}
Note that Lemma~\ref{lem:MinimumDistance-By-MinimumExcludeMult} can be generalized to Sidon sets with the same bound except the $+1$ term.
Let $S,T \subseteq \F_2^n$ be distinct Sidon sets.
Let $y \in T$ such that $y \notin S$, and let $k = \mult_S(y)$.
By definition, there exist distinct (and necessarily pairwise disjoint) triples $\set{a_i, b_i, c_i}_{i=1}^k$ such that $a_i, b_i, c_i \in S$ and $a_i+b_i+c_i = y$. 
Note that $\set{a_i, b_i, c_i}$ cannot be a subset of $T$ for any $1 \leq i \leq k$ because this would contradict $T$ being Sidon.
Hence, there are at least $k$ points in $S$ that do not lie in $T$, and so $|S\setminus T| \geq k \geq e_{\min}(S)$.
We cannot improve this bound to also include the $+1$ term as there exist distinct Sidon sets $S,T \subseteq \F_2^n$ such that $|S \setminus T| = e_{\min}(S)$.
For example, taking $S_0$ to be a maximal Sidon set in $\F_2^4$ of size 6, and then taking $S$ and $T$ to be any two distinct extensions of $S_0$ in $\F_2^5$ yields two distinct maximal Sidon sets of $\F_2^5$ of size 7 (for information on the affine equivalence classes of Sidon sets in at most six dimensions, see \cite{quadspaper}).
This implies $S$ and $T$ intersect at six points yet both have minimum exclude multiplicity $1$.
\end{remark}

Let us now recall some notation and definitions that were first introduced in \cite{DistanceBetweenAPNFunctions2020}.
Let $F \colon \F_2^n \to \F_2^n$ be a vectorial Boolean function.
The \textit{shifted derivative $D_a^\beta F$ of $F$ in direction $a$ with shift $\beta$} is the function 
\[
    D_a^\beta F = D_aF(x) +F(a+\beta) = F(x)+F(a+x)+F(a+\beta).
\]
For an $(n,n)$-function $F$, we denote the image of $D_a^\beta F$ as $H_a^\beta F$.
Moreover, we denote by $\Pi_F^\beta(b)$ the set $\Pi_F^\beta(b) = \set{a \in \F_2^n : b \in H_a^\beta F}$.
Also, let 
\[
\Pi_F^\beta = \set{|\Pi_F^\beta(b)| : b \in \F_2^n},
\]
and let 
\[
\Pi_F = \bigcup_{\beta \in \F_2^n} \Pi_F^\beta = \set{|\Pi_F^\beta (b)| : \beta, b \in \F_2^n}.
\]

We now compute the size of $\Pi_F^\beta(b)$ in various ways, and this will be useful in demonstrating the connection between shifted derivatives and exclude multiplicity. 

\begin{proposition}\label{prop:PiF-beta-excludemults}
    Let $F \colon \F_2^n \to \F_2^n$ be an APN function.
    For any $(b, \beta) \in (\F_2^n)^2$, we have the equalities 
    \begin{align*}
    |\Pi_F^\beta(b)| &= \delta_0(b + F(\beta)) + \sum_{a \in \F_2^n} \gamma_F(a+\beta, F(a)+b)\\
        &=\frac{1}{2^{2n+1}}\widehat{W_F^3}(\beta, b)-\delta_0(b+F(\beta))(2^{n-1}-1).
    \end{align*}
    In particular, if $b \neq F(\beta)$, then 
    \[
    |\Pi_F^\beta(b)|  =\sum_{a \in \F_2^n} \gamma_F(a+\beta, F(a)+b) = \frac{1}{2^{2n+1}}\widehat{W_F^3}(\beta, b) =  3\mult_{\graph F}(\beta, b).
    \]
\end{proposition}
\begin{proof}
    Notice that for any $\beta, b \in \F_2^n$ , we have
    \begin{align*}
       \Pi_F^\beta(b) &=\set{a \in \F_2^n : b \in H_a^\beta F} \\
         &= \set{a \in \F_2^n : b \in \im (D_aF) + F(a+\beta)}  \\
         &= \set{a \in \F_2^n : \delta_F(a,b +F(a+\beta)) \neq 0}.
    \end{align*}
    Therefore the size of $\Pi_F^\beta(b)$ is the same as 
    \[
    \delta_0(b + F(\beta)) + \sum_{a \in \F_2^n} \gamma_F(a,b + F(a + \beta)) 
    =
    \delta_0(b + F(\beta)) + \sum_{a \in \F_2^n} \gamma_F(a+\beta, F(a)+b).
    \]
    Hence $|\Pi_F^\beta(b)| = \delta_0(b + F(\beta)) + \sum_{a \in \F_2^n} \gamma_F(a+\beta, F(a)+b)$.
    Since $F$ is APN, we know that $\gamma_F = \frac{1}{2}\parens{\delta_F - 2^n \delta_{(0,0)}}$, and recall that we also have the equality $\delta_F = 1_{\graph F} \otimes 1_{\graph F}$.
    Therefore
    \begin{align*}
        \sum_{a \in \F_2^n} \gamma_F(a+\beta, F(a)+b) &= (\gamma_F \otimes 1_{\graph F})(\beta, b) \\
        &= \frac{1}{2} ((\delta_F - 2^n\delta_0) \otimes 1_{\graph F})(\beta, b) \\
        &= \frac{1}{2} ((1_{\graph F} \otimes 1_{\graph F}) \otimes 1_{\graph F})(\beta, b) - 2^{n-1}1_{\graph F}(\beta, b)
    \end{align*}
    The convolution law gives $\widehat{\varphi \otimes \psi} = \widehat{\varphi} \widehat{\psi}$ for any functions $\varphi, \psi \colon \F_2^n \to \Z$, and using the fact that $\widehat{1_{\graph F}}= W_F$, the Fourier-Hadamard transform of $(1_{\graph F} \otimes 1_{\graph F}) \otimes 1_{\graph F}$ is $\widehat{(1_{\graph F} \otimes 1_{\graph F})} \widehat{1_{\graph F}} =(\widehat{1_{\graph F}} \widehat{1_{\graph F}})\widehat{1_{\graph F}}= W_F^3$. 
    Therefore, taking the Fourier-Hadamard transform of both sides, we have $(1_{\graph F} \otimes 1_{\graph F}) \otimes 1_{\graph F} = \frac{1}{2^{2n}}\widehat{W_F^3}$, and so
    \begin{align*}
        |\Pi_F^\beta(b)| &= \delta_0(b + F(\beta))  + \sum_{a \in \F_2^n} \gamma_F(a+\beta, F(a)+b) \\
        &= \delta_0(b + F(\beta)) + \frac{1}{2} ((1_{\graph F} \otimes 1_{\graph F}) \otimes 1_{\graph F})(\beta, b) - 2^{n-1}1_{\graph F}(\beta, b) \\
        &=  \frac{1}{2^{2n+1}}\widehat{W_F^3}(\beta, b)-\delta_0(b+F(\beta))(2^{n-1}-1).
    \end{align*}
    In particular, when $(\beta, b) \notin \mathcal{G}_F$, we have 
    \[ 
    |\Pi_F^\beta(b)| = \frac{1}{2^{2n+1}}\widehat{W_F^3}(\beta, b)= 3\mult_{\graph F}(\beta, b)
    \]
    by applying eq.~(\ref{eq:WFk-set-size}).
\end{proof}

Therefore, for an APN function $F \colon \F_2^n \to \F_2^n$, the exclude points of $\graph F$ and their multiplicities are fully described by the sizes of the sets $\Pi_F^\beta(b)$.
Note that the set $\Pi_F^\beta$ describes the exclude multiplicities of points in  $\set{\beta} \times \F_2^n$.
We have the following corollary of Proposition~\ref{prop:PiF-beta-excludemults}.
\begin{corollary}\label{cor:PiFbeta-excludemults}
    Let $F \colon \F_2^n \to \F_2^n$ be an APN function.
    Then 
    \[
    \Pi_F^\beta = \set{2^n} \cup \set{3\mult_{\graph F}(\beta, b) : b \in \F_2^n \setminus \set{F(\beta)}}.
    \]
\end{corollary}

If $F$ is APN, then the set $\Pi_F \setminus \set{2^n}$ is the set of all exclude multiplicities of $\graph F$ scaled by $3$.
In \cite{DistanceBetweenAPNFunctions2020}, it was shown that $\Pi_F$ can be used to derive a lower bound on the distance from an APN function $F$ to any other APN function.

\begin{corollary}[\textup{\cite{DistanceBetweenAPNFunctions2020}}]
\label{cor:MinimumDistance-mF}
    Let $F \colon \F_2^n \to \F_2^n$ be an APN function, and let $m_F$ be the number 
    \begin{align*}
        m_F = \min \Pi_F = \min_{b,\beta \in \F_2^n}|\Pi_F^\beta(b)|.
    \end{align*}
    Then for any APN function $G \neq F$ over $\F_2^n$, the Hamming distance between $F$ and $G$ satisfies 
    \[
    d(F,G) \geq \ceil{\frac{m_F}{3}} +1.
    \]
\end{corollary}
For any APN function $F$, any exclude point of $\graph F$ has exclude multiplicity at most $\lfloor \frac{2^n}{3}\rfloor$ (see \cite[Corollary 3.4]{quadspaper}).
Therefore, $m_F = \min\set{3e_{\min}(\graph F), 2^n} = 3e_{\min}(\graph F)$, and so the bound of Corollary~\ref{cor:MinimumDistance-mF} is equivalent to the bound of Lemma~\ref{lem:MinimumDistance-By-MinimumExcludeMult}.

\section{Plateaued APN functions}\label{sec:plateaued}

To provide a brief overview of this section, we study the exclude multiplicities of $\graph F$ when $F$ is a plateaued APN function and consider particular subclasses of plateaued APN functions such as AB functions, plateaued $3$-to-$1$ functions, and quadratic APN functions.

\subsection{AB functions}\label{subsec:AB-functions}

The first class of plateaued APN functions that we consider is the class of almost bent (AB) functions.
A vectorial Boolean function $F \colon \F_2^n \to \F_2^n$ is AB if $W_F(u,v) \in \set{0, \pm 2^{\frac{n+1}{2}}}$ for all nonzero $(u,v) \in (\F_2^n)^2$, and so AB functions only exist when $n$ is odd.
Moreover, all plateaued APN functions are necessarily AB when $n$ is odd.
We can easily derive a lower bound on the distance from any AB function to another APN function using the van Dam and Fon-Der-Flaass characterization of AB functions.
\begin{theorem}[\cite{vandamflass}]\label{thm:AB-vanDamFlaass}
    Let $F \colon \F_2^n \to \F_2^n$ be a function.
    Then $F$ is AB if and only if the system of equations
    \[
        \begin{cases}
            x + y + z = a \\
            F(x) + F(y) + F(z) = b
        \end{cases}
    \]
    has $2^n -2$ or $3 \cdot 2^n - 2$ solutions $(x,y,z) \in (\F_2^n)^3$ for every $(a,b) \in (\F_2^n)^2$.
    If so, then the system has $2^n -2$ solutions if $b \neq F(a)$ and $3 \cdot 2^n - 2$ solutions otherwise.
\end{theorem}

This useful characterization immediately shows that the exclude multiplicity of any point in $(\F_2^n)^2 \setminus \graph F$ is $\frac{2^n-2}{6}$ when $F \colon \F_2^n \to \F_2^n$ is an AB function.
The following result immediately follows and was first stated by Coulter and Kaleyski in \cite{CoulterKaleyski2021}.

\begin{theorem}\label{thm:dist-to-AB-functions}
    Suppose $F \colon \F_2^n \to \F_2^n$ is an AB function, and let $G \colon \F_2^n \to \F_2^n$ be an APN function such that $G \neq F$.
    Then 
    \[
        d(F,G) \geq \frac{2^n-2}{6}+1 = \frac{2^{n-1}+2}{3}.
    \]
\end{theorem}
\begin{proof}
    Apply Lemma~\ref{lem:MinimumDistance-By-MinimumExcludeMult} to Theorem~\ref{thm:AB-vanDamFlaass}.
\end{proof}

\subsection{Properties of the exclude multiplicities of $\graph F$ when $F$ is plateaued APN}\label{subsec:plateaued-properties}

To begin, let us first recall the following characterization of plateaued vectorial functions.

\begin{theorem}[\textup{\cite{CarletPlateaued}}]\label{thm:Carlet-PlateauedCharacterization}
    Let $F \colon \F_2^n \to \F_2^m$ be a vectorial function.
    Then
    \begin{enumerate}
        \item $F$ is plateaued if and only if, for every $w \in \F_2^m$, the size of 
    \begin{equation}\label{eq:plateaued-characterization-by-set}
        \set{(a,b) \in (\F_2^n)^2 : D_aD_bF(x)=w}
    \end{equation}
    does not depend on $x \in \F_2^n$;
    \item $F$ is plateaued with single amplitude if and only if the size of the set in eq.~(\ref{eq:plateaued-characterization-by-set}) does not depend on $x \in \F_2^n$, nor on $w \in \F_2^m$ when $w \neq 0$.
    \end{enumerate}
\end{theorem}

Hence, for a plateaued function $F \colon \F_2^n \to \F_2^n$, the size of the set 
\[
\set{(a,b) \in (\F_2^n)^2 : F(x)+F(x+a)+F(x+b)+F(x+a+b)=w}
\]
is independent of the choice of $x \in \F_2^n$ when $w \in \F_2^n$ is fixed.
By replacing $b$ with $x+b$, the size of the above set is the same as 
\begin{align*}
    &|\set{(a,b) \in (\F_2^n)^2 : F(x)+F(x+a) + F(b)+F(a+b) = w}|\\
         &= \sum_{a \in \F_2^n} |\set{b \in \F_2^n : F(x)+F(x+a)+F(b)+F(a+b)=w}| \\ 
         &= \sum_{a \in \F_2^n} |\set{b \in \F_2^n : D_aF(b) = F(x)+F(x+a)+w}| \\
          &= \sum_{a \in \F_2^n} \delta_F(a, F(x)+F(x+a)+w).
\end{align*}
By replacing $a$ by $x+a$, we then have 
\[
|\set{(a,b) \in (\F_2^n)^2 : D_a D_b F(x)=w}| = \sum_{a \in \F_2^n} \delta_F(a+x, F(a)+F(x)+w).
\]
Let us consider the case that $F$ is APN.
Then $\delta_F = 2\gamma_F + 2^n \delta_{(0,0)}$.
So, if $w \neq 0$, then $\sum_{a \in \F_2^n} \delta_F(a+x, F(a)+F(x)+w)= 2\sum_{a \in \F_2^n} \gamma_F(a+x, F(a)+F(x)+w)$, and by applying Proposition~\ref{prop:PiF-beta-excludemults} with $b = F(x)+w$ and $\beta = x$, we know that $\sum_{a \in \F_2^n} \gamma_F(a+x, F(a)+F(x)+w) = 3\mult_{\graph F}(x,F(x)+w)$.
Hence 
\begin{equation}\label{eq:second-order-deriv-excludemult}
    |\set{(a,b) \in (\F_2^n)^2 : D_a D_b F(x)=w}| = 6 \mult_{\graph F}(x, F(x)+w)
\end{equation}
for any $w \neq 0$.
So Theorem~\ref{thm:Carlet-PlateauedCharacterization} implies that if $F$ is plateaued and $w \neq 0$, then $\mult_{\graph F}(x, F(x)+w)$ does not depend on $x \in \F_2^n$.
For $F \colon \F_2^n \to \F_2^n$ quadratic, it was shown in  \cite[Proposition 5]{DistanceBetweenAPNFunctions2020} that $\Pi_F^\beta$ does not depend on $\beta \in \F_2^n$, and as a consequence, the values that $\mult_{\graph F}(a,b)$ takes as $b$ ranges across $\F_2^n \setminus \set{F(a)}$ do not depend on the choice of $a \in \F_2^n$ by Proposition~\ref{prop:PiF-beta-excludemults}.
In the following proposition, we prove this result for all plateaued APN functions.

\begin{proposition}\label{prop:plateaued-uniformity}
    Let $F \colon \F_2^n \to \F_2^n$ be a plateaued APN function.
    If $a,b,c \in \F_2^n$ such that $b \neq F(a)$, then $\mult_{\graph F}(a,b) = \mult_{\graph F}(c,b+F(a)+F(c))$. 
    In particular, the exclude multiplicities of points in $\set{a} \times (\F_2^n \setminus \set{F(a)})$, with respect to $\graph F$, do not depend on the choice of $a$.
\end{proposition}
\begin{proof}
    For any $w \in \F_2^n \setminus \set{0}$, since $F$ is plateaued APN, we know from Theorem~\ref{thm:Carlet-PlateauedCharacterization} and eq.~(\ref{eq:second-order-deriv-excludemult}) that the value of $\mult_{\graph F}(x, F(x)+w)$ does not depend on the choice of $x \in \F_2^n$.
    For all $a \in \F_2^n$, let 
    \[
    X_a  = \set{a} \times (\F_2^n \setminus \set{F(a)}),
    \]
    and define $\pi_{a,0} \colon X_a \to X_0$ to be the permutation given by $\pi_{a,0}(a,b) = (0,b+F(a)+F(0))$ for all $b \neq F(a)$.
    Then $\pi_{a,0}$ also preserves exclude multiplicities because for any $(a,b) \notin \graph F$ with $w =F(a)+b \neq 0$, we have
    \begin{align*}
        \mult_{\graph F}(a,b)&=\mult_{\graph F}(a,F(a)+w) = \mult_{\graph F}(0,F(0)+w)
        = \mult_{\graph F}(0, F(0)+F(a)+b) \\
        &= \mult_{\graph F}(\pi_{a,0}(a,b)). 
    \end{align*}
    Note that the inverse of $\pi_{a,0}$ is given by $(0,b) \mapsto (a, F(a)+F(0)+b)$.
    For all $a,c \in \F_2^n$, let $\pi_{a,c} \colon X_a \to X_c$ be the permutation $\pi_{a,c}= \pi_{c,0}^{-1} \circ \pi_{a,0}$.
    Then $\mult_{\graph F}(a,b) = \mult_{\graph F}(\pi_{a,c}(a,b)) = \mult_{\graph F}(c,b+F(a)+F(c))$ for all $a,b,c \in \F_2^n$ with $b \neq F(a)$.
\end{proof}

We then immediately have the following divisibility condition on the frequencies of exclude multiplicities as a corollary of Proposition~\ref{prop:plateaued-uniformity}.
\begin{corollary}\label{cor:mult-divis}
    Let $F \colon \F_2^n \to \F_2^n$ be a plateaued APN function.
    For any non-negative integer $k$, let $m_k = |\set{(a,b) \in (\F_2^n)^2 : \mult_{\graph F}(a,b) =k}|$.
    Then $2^n|m_k$ for all $k \geq 0$. 
\end{corollary}
\begin{proof}
    For any non-negative integer $k$, let $z_k = |\set{b \in \F_2^n \setminus \set{F(0)} : \mult_{\graph F}(0,b) =k}|$.
    By Proposition~\ref{prop:plateaued-uniformity}, we know that for any $a \in \F_2^n$ there exists a bijection from $\set{a} \times (\F_2^n \setminus \set{F(a)})$ to $\set{0} \times (\F_2^n \setminus \set{F(0)})$ that preserves exclude multiplicity.
    Therefore, $m_k = 2^n z_k$ for all $k \geq 0$, so $2^n | m_k$.
\end{proof}

From this, we prove that the Brinkmann-Leander-Edel-Pott function $F \colon \F_{2^6} \to \F_{2^6}$ (see \cite{EdelPottSporadic}, \cite[Section 11.5.3]{CarletBook}, or \cite{Brinkmann2008}), given by 
\begin{align*}
    F(x)&=x^3 
    + u^{17} (x^{17} + x^{18} + x^{20}+x^{24}) 
    + u^{14} (u^{18} x^9 + u^{36}x^{18} + u^9x^{36} + x^{21} + x^{42}) \\
    &+ u^{14}\Tr(u^{52} x^3 + u^6x^5 + u^{19} x^7 + u^{28} x^{11} + u^2x^{13}),
\end{align*}
where $u \in \F_{2^n}$ is primitive, cannot be CCZ-equivalent to a plateaued function.
This improves our understanding of this sporadic APN function as the Brinkmann-Leander-Edel-Pott function was only previously known to not be CCZ-equivalent to a quadratic function or monomial function \cite{EdelPottSporadic}.
\begin{corollary}
    The Brinkmann-Leander-Edel-Pott function is not CCZ-equivalent to any plateaued APN function.
\end{corollary}
\begin{proof}
    Let $F \colon \F_{2^6} \to \F_{2^6}$ be the Brinkmann-Leander-Edel-Pott function.
    A straightforward computation shows the points of $\F_{2^6}^2 \setminus \graph{F}$ have exclude multiplicities $5,7,9,11,13$, and $15$ with frequencies $40,360,1296,1616,648$, and $72$, respectively (one can compute this by computing the Fourier-Hadamard transform of $W_F^3$, see eq.~(\ref{eq:WFk-set-size})).
    By Corollary~\ref{cor:mult-divis} and the fact that exclude multiplicities are preserved under CCZ-equivalence, any function CCZ-equivalent to a plateaued function must satisfy the property that the frequency of any exclude multiplicity must be divisible by $2^n$.
    Hence, the Brinkmann-Leander-Edel-Pott function is not CCZ-equivalent to any plateaued APN function as $2^6$ does not divide $40$ (or any other frequency listed above).
\end{proof}

Our proof technique for showing that the Brinkmann-Leander-Edel-Pott function is not CCZ-equivalent to a plateaued function also yields a practical test that can be applied to newly discovered APN functions.
In general, determining whether or not an arbitrary APN function $F$ is CCZ-equivalent to a plateaued or quadratic (or more generally, crooked) function is a difficult problem.
In \cite{EdelPottSporadic}, it was shown that the Brinkmann-Leander-Edel-Pott function is not CCZ-equivalent to a quadratic APN function (more generally, they proved it was not CCZ-equivalent to a crooked function) by proving that an associated so-called $\Delta$-rank was too large.

We can also describe the exclude multiplicities of points in $(\F_2^n)^2 \setminus \graph F$ with respect to $\graph F$ more directly in terms of the amplitudes of the component functions of $F$.

\begin{proposition}\label{prop:plateaued-exclude-multiplicity}
    Let $F \colon \F_2^n \to\F_2^n$ be a plateaued APN function, and let $\lambda_v$ denote the amplitude of $v \cdot F$.
    Then for any $(a,b) \in (\F_2^n)^2 \setminus \graph F$, we have
    \begin{equation}\label{eq:plateaued-exclude-multiplicities}
    \mult_{\graph F}(a,b) =\frac{1}{6 \cdot 2^n} \sum_{v \in \F_2^n}(-1)^{v \cdot (F(a) +b )}\lambda_v^2.
    \end{equation}
\end{proposition}
\begin{proof}
    For any $(u,v) \in (\F_2^n)^2 \notin \graph F$, we have $W_F^3(u,v)=\lambda_v^2 W_F(u,v)$ as $F$ is plateaued.
    Recall that for any $(a,b) \notin \graph F$, we have $\mult_{\graph F}(a,b) = \frac{1}{6 \cdot 2^{2n}}\widehat{W_F^3}(a,b)$ by eq.~(\ref{eq:WFk-set-size}), and by Proposition~\ref{prop:plateaued-uniformity}, we also have that $\mult_{\graph F}(a,b) = \mult_{\graph F}(0,b+F(a)+F(0))$.
    So, for $(a,b) \notin \graph F$, 
    \begin{align*}
        \mult_{\graph F}(a,b) &=\frac{1}{6 \cdot 2^{2n}}\sum_{(u,v) \in (\F_2^n)^2}(-1)^{v\cdot (b+F(a)+F(0))} \lambda_v^2 W_F(u,v)\\
        &= \frac{1}{6 \cdot 2^{2n}}\sum_{v \in \F_2^n}(-1)^{v\cdot (b+F(a)+F(0))} \lambda_v^2 \sum_{x \in \F_2^n} (-1)^{v \cdot F(x)} \sum_{u \in \F_2^n}  (-1)^{x \cdot u}\\
        &= \frac{1}{6 \cdot 2^n}\sum_{v \in \F_2^n}(-1)^{v\cdot (b+F(a))} \lambda_v^2.
    \end{align*}
\end{proof}

As seen in Proposition~\ref{prop:plateaued-exclude-multiplicity}, the amplitudes of the component functions of a plateaued APN function can be used to describe the exclude multiplicities of points in $(\F_2^n)^2 \setminus \graph{F}$ with respect to $\graph F$.
For $c \neq 0$, it is clear that we have the equalities 
\begin{equation}\label{eq:amplitudes-rephrased}
2\sum_{v \in \set{0,c}^\perp} \lambda_v^2 - \sum_{v \in \F_2^n} \lambda_v^2
= \sum_{v \in \F_2^n}(-1)^{v \cdot c} \lambda_v^2 
= \sum_{v \in \F_2^n} \lambda_v^2 - 2\sum_{v \notin \set{0,c}^\perp} \lambda_v^2.
\end{equation}
By using the following proposition, we are able to rephrase bounding from below $e_{\min}(\graph F)$, that is $\min_{(a,b) \in (\F_2^n)^2 \setminus \graph{F}}\mult_{\graph F}(a,b)$, in terms of minimizing the sum of squared amplitudes over any linear hyperplane.
This is because $\sum_{v \in \F_2^n} \lambda_v^2=2^n(3\cdot 2^n-2)$ precisely when $F$ is APN.

\begin{proposition}[{\cite[Proposition 9]{CarletPlateaued}}]\label{prop:Carlet-amplitude-squared-characterization}
    Let $F \colon \F_2^n \to \F_2^n$ be a plateaued function, and let $\lambda_v$ denote the amplitude of $v \cdot F$ for all $v \in \F_2^n$.
    Then $F$ is APN if and only if 
    \begin{equation}\label{eq:APN-iff-amplitudes}
    \sum_{v\in \F_2^n}\lambda_v^2 \leq 2^n(3\cdot 2^n-2).
    \end{equation}
    In particular, if $F$ is APN, then (\ref{eq:APN-iff-amplitudes}) is an equality.
\end{proposition}

\begin{corollary}\label{cor:mult-hyperplane}
    Let $F \colon \F_2^n \to \F_2^n$ be a plateaued APN function, and let $\lambda_v$ denote the amplitude of $v \cdot F$ for all $v \in \F_2^n$.
    Then for any $(a,b) \in (\F_2^n)^2 \setminus \graph F$, we have
    \begin{align*}
        \mult_{\graph F}(a,b) = \frac{1}{3 \cdot 2^n} \parens{\sum_{v \in \set{0,b+F(a)}^\perp} \lambda_v^2 - 2^{n-1}(3\cdot 2^n-2)}.
    \end{align*}
\end{corollary}
\begin{proof}
    Let $(a,b) \in (\F_2^n)^2 \setminus \graph F$.
    By eq.~(\ref{eq:plateaued-exclude-multiplicities}), we have 
    \[
    \mult_{\graph F}(a,b) =\frac{1}{6 \cdot 2^n} \sum_{v \in \F_2^n}(-1)^{v \cdot (F(a) +b )}\lambda_v^2 = \frac{1}{6 \cdot 2^n} \parens{2\sum_{v \in \set{0,b + F(a)}^\perp} \lambda_v^2 - \sum_{v \in \F_2^n}\lambda_v^2}.
    \]
    The result then follows by applying Proposition~\ref{prop:Carlet-amplitude-squared-characterization}.
\end{proof}

\begin{corollary}\label{cor:plat-dist-by-hyperplane}
   Let $F \colon \F_2^n \to \F_2^n$ be a plateaued APN function, and let $\lambda_v$ denote the amplitude of $v \cdot F$ for all $v \in \F_2^n$.
   For any $(a,b) \in (\F_2^n)^2 \setminus \graph F$, we have 
    \begin{equation}\label{eq:minimum-multiplicity-in-terms-of-hyperplane}
    \sum_{v \in \set{0,b+F(a)}^\perp} \lambda_v^2 = 3 \cdot 2^n \mult_{\graph F}(a,b) + 2^{n-1}(3 \cdot 2^n-2).
    \end{equation}
\end{corollary}
\begin{proof}
    Immediate upon rearrangement of the equation from Corollary~\ref{cor:mult-hyperplane}.
\end{proof}

Although Corollary~\ref{cor:plat-dist-by-hyperplane} follows easily from previous statements, it helps shape our approach on finding a lower bound on the minimum exclude multiplicity of the graph of a plateaued APN function.
By using a geometric constraint on the structure of $\bentcomps{F} = \set{v \in \F_2^n : v \cdot F \text{ is bent}}$, we can determine a lower bound on the distance between plateaued APN functions and other APN functions.
We use the following corollary which follows from a classical result of \cite{BOSE196696}.

\begin{corollary}[{\cite[Corollary 1]{MaxNumBentComponents}}]\label{cor:set-intersecting-subspaces}
    A set $S \subseteq (\F_2^n \setminus \set{0})$ that intersects every $(n+1-k)$-dimensional subspace of $\F_2^n$ has at least $2^k-1$ elements with equality if and only if $S \cup \set{0}$ is a $k$-dimensional subspace of $\F_2^n$.
\end{corollary} 

The following lemma is a simple generalization of Theorem 3 from \cite{MaxNumBentComponents}.

\begin{lemma}\label{lem:nontrivial-nonbent-components}
Assume $n$ is even.
    Let $F \colon \F_2^n \to \F_2^n$ be any vectorial Boolean function.
    Let $S \subseteq \F_2^n$ be a linear subspace of dimension $d$ where $\frac{n}{2}+1 \leq d \leq n$.
    Then $|S \setminus \bentcomps{F}| \geq 2^{d-\frac{n}{2}}$ with equality if and only if $S\setminus \bentcomps{F}$ is a $(d-\frac{n}{2})$-dimensional linear subspace of $S$.
    Hence, there are at least $2^{d-\frac{n}{2}}-1$ nonzero, non-bent component functions $v \cdot F$ with $v \in S$.
\end{lemma}
\begin{proof} 
    In \cite[Proof of Theorem 3]{MaxNumBentComponents}, it was shown that there cannot exist a $(\frac{n}{2}+1)$-dimensional linear subspace $T$ of $\F_2^n$ such that $T \setminus \set{0}$ is contained in $\bentcomps{F}$ (otherwise we could easily construct a function $\F_2^n \to \F_2^{n/2+1}$ whose component functions are all bent, but this is impossible by Nyberg's bound).
    Hence, $S \cap \bentcomps{F}$ cannot contain all nonzero elements of a linear subspace of dimension $\frac{n}{2}+1$ because it is a subset of $\bentcomps{F}$.
    So $S \setminus (\bentcomps{F}\cup \set{0})$ must intersect every linear subspace of $S$ of dimension $\frac{n}{2}+1$.
    Also, note that $\frac{n}{2}+1 = d+1 - (d-\frac{n}{2})$.
    So, by applying a linear isomorphism from $S$ to $\F_2^d$ and considering the image of $S \setminus (\bentcomps{F}\cup \set{0})$, it follows from Corollary~\ref{cor:set-intersecting-subspaces} that 
    \[
    |S \setminus (\bentcomps{F}\cup \set{0})| \geq 2^{d-\frac{n}{2}}-1
    \]
    with equality if and only if $S \setminus \bentcomps{F}$ is a $(d-\frac{n}{2})$-dimensional subspace of $S$.
\end{proof}

From this, we are able to derive lower bounds on the exclude multiplicities of the graphs of plateaued APN functions (for $n$ even). 

\begin{theorem}\label{thm:plat-exclude-bound}
    Assume $n$ is even.
    Let $F \colon \F_2^n \to \F_2^n$ be a plateaued APN function.
    Then $\mult_{\graph F}(a,b) \geq 2^{\frac{n}{2}-1}-1$ for all $(a,b) \in (\F_2^n)^2 \setminus \graph F$.
\end{theorem}
\begin{proof}
    Let $c = F(a) + b$ for $(a,b) \in (\F_2^n)^2 \setminus \graph F$.
    Let $N = |\set{0,c}^\perp \setminus (\bentcomps{F}\cup \set{0})|$.
    Recall that if $v \in \bentcomps{F}$, then $\lambda_v=2^{\frac{n}{2}}$, and otherwise, $\lambda_v \geq 2^{\frac{n}{2}+1}$. 
    So
    \begin{align*}
    \sum_{v \in \set{0,c}^\perp} \lambda_v^2 
    &= 2^{2n} + (2^{n-1} - 1 -N)2^n + \sum_{v \in \set{0,c}^\perp \setminus (\bentcomps{F} \cup \set{0})} \lambda_v^2 \\
    &\geq 2^{2n} + (2^{n-1} - 1 -N)2^n + N 2^{n+2} \\
        &= 3 \cdot 2^n N + 2^{n-1}(3 \cdot 2^n-2).
    \end{align*}
    Applying eq.~(\ref{eq:minimum-multiplicity-in-terms-of-hyperplane}) to the left-hand-side of the inequality above, we have that $\mult_{\graph F}(a,b) \geq N$.
    By Lemma~\ref{lem:nontrivial-nonbent-components}, we have $N \geq 2^{\frac{n}{2}-1}-1$, and so $\mult_{\graph F}(a,b) \geq 2^{\frac{n}{2}-1}-1$.
\end{proof}

Our lower bound in Theorem~\ref{thm:plat-exclude-bound} is a dramatic increase over the previously known $e_{\min}(\graph F) = \min_{(a,b) \in (\F_2^n)^2 \setminus \graph F}\mult_{\graph F}(a,b) \geq 1$ (for $n\geq 3$) when $F$ is plateaued APN.
It would be interesting to further improve this lower bound or find an example of a family of plateaued APN functions for which it is tight.
The first of our three main results immediately follows.

\begin{proof}[Proof of Theorem~\ref{thm:plat-dist-bound}]
    For any plateaued APN function $F \colon \F_2^n \to \F_2^n$ and any point $(a,b)$ of $(\F_2^n)^2$ not contained in $\graph F$, we have that $\mult_{\graph F}(a,b) \geq 2^{\frac{n}{2}-1}-1$ by Theorem~\ref{thm:plat-exclude-bound}.
    So, by Lemma~\ref{lem:MinimumDistance-By-MinimumExcludeMult}, we have
    \[
    d(F,G) \geq e_{\min}(\graph F)+1 
    = \min_{(a,b) \in (\F_2^n)^2 \setminus \graph F}\mult_{\graph F}(a,b) + 1 
    \geq 
    2^{\frac n 2 -1}
    \]
    for any APN function $G \neq F$.
\end{proof}

When $n$ is odd, any plateaued APN function $F \colon \F_2^n \to \F_2^n$ is AB, and so we know by a result of \cite{vandamflass} that $\graph F$ is a maximal Sidon set with all points outside of the graph having exclude multiplicity $\frac{2^n-2}{6}$ which is odd (see Section~\ref{subsec:AB-functions}).
However, when $n$ is even, the exclude multiplicities of $\graph F$ are much more unclear.
In the following proposition, we describe the parity of the exclude multiplicities of $\graph F$ for $F$ plateaued APN in terms of the Fourier transform of the indicator function $1_\bentcomps{F} \colon \F_2^n \to \set{0,1}$ of $\bentcomps{F}$.
This then shows that the exclude multiplicities of the graph of any plateaued APN function are always odd (for $n \geq 3$ since $\graph F$ is not a maximal Sidon set if $n \leq 2$).

\begin{proposition}\label{prop:plateaued-odd-mults}
    Assume $n \geq 3$.
    Let $F \colon \F_2^n \to \F_2^n$ be a plateaued APN function, and let $(a,b) \in (\F_2^n)^2 \setminus \graph F$.
    Then $\mult_{\graph F}(a,b)$ is odd.
\end{proposition}
\begin{proof}
    If $n$ is odd, then $F$ is AB and $\mult_{\graph F}(a,b)$ is odd.
    So, assume $n$ is even.
    We claim that $\mult_{\graph F}(a,b)$ is odd if and only if $\widehat{1_\bentcomps{F}}(F(a)+b) \equiv 2 \mod 4$.
    For all $v \in \F_2^n$, let $\lambda_v$ denote the amplitude of $v \cdot F$, and let $m_v$ be the non-negative integer such that $\lambda_v^2 = 2^{n+m_v}$.
    Note that $m_v$ is even since $\lambda_v$ is a power of $2$ and $n$ is even.
    Moreover, $0 \leq m_v \leq n$ for all $v \in \F_2^n$.
    Let $(a,b) \in (\F_2^n)^2 \setminus \graph F$, and let $c = F(a) +b$.
    By Proposition~\ref{prop:plateaued-exclude-multiplicity}, we have
    \[
    6 \mult_{\graph F}(a,b) = \frac{1}{2^n} \sum_{v \in \F_2^n}(-1)^{v \cdot c}  2^n 2^{m_v} = \sum_{v \in \F_2^n}(-1)^{v \cdot c}2^{m_v}.
    \]
    Since $\mult_{\graph F}(a,b)$ is an integer, we then know that $\mult_{\graph F}(a,b)$ is odd if and only if 
    \[
    \sum_{v \in \F_2^n}(-1)^{v \cdot c}2^{m_v} \equiv 6 \mod 12.
    \]
    Also, $2^{m_v} \equiv 1 \mod 12$ if $m_v=0$ (i.e. $v \cdot F$ is bent) and $2^{m_v} \equiv 4 \mod 12$ if $m_v \geq 2$ as $m_v$ is even. 
    For $v \in \F_2^n \setminus \bentcomps{F}$, we then can write $2^{m_v} = 12 r_v + 4$ for some non-negative integer $r_v$, so $r_v = \frac{2^{m_v}-4}{12}$.
    Hence
    \begin{align*}
    6 \mult_{\graph F}(a,b) &= \sum_{v \in \bentcomps{F}}(-1)^{v \cdot c} + \sum_{v \in  \F_2^n \setminus \bentcomps{F}}(-1)^{v \cdot c}( 12 r_v + 4)\\
    &= \widehat{1_\bentcomps{F}}(c) + 4 \cdot \widehat{1_{\F_2^n \setminus \bentcomps{F}}}(c) + 12 \sum_{v \in  \F_2^n \setminus \bentcomps{F}} (-1)^{v \cdot c} r_v.
    \end{align*}
    So $6 \mult_{\graph F}(a,b) \equiv \widehat{1_\bentcomps{F}}(c) + 4 \cdot \widehat{1_{\F_2^n \setminus \bentcomps{F}}}(c) \mod 12$.
    Note that 
    $
    \widehat{1_\bentcomps{F}}(c) + 4 \cdot \widehat{1_{\F_2^n \setminus \bentcomps{F}}}(c)
    \equiv 6 \mod 12
    $
    if and only if $\widehat{1_\bentcomps{F}}(c) \equiv 2 \mod 4$ and $\widehat{1_\bentcomps{F}}(c) + \widehat{1_{\F_2^n \setminus \bentcomps{F}}}(c)\equiv 0 \mod 3$.
    However, we already know that $\widehat{1_\bentcomps{F}}(c) + \widehat{1_{\F_2^n \setminus \bentcomps{F}}}(c)=0$ as $c \neq 0$.
    So $\mult_{\graph F}(a,b)$ is odd if and only if $\widehat{1_\bentcomps{F}}(c) \equiv 2 \mod 4$, and this proves our claim.

    Now, for $i \in \set{0,1}$, let $B_i = |\set{v \in \bentcomps{F} : v \cdot c = i}|$.
    Then $\widehat{1_\bentcomps{F}}(c) = |B_0|-|B_1| = |\bentcomps{F}|-2|B_1|$.
    It was shown in \cite[Proposition 6.5]{kolsch2025combinatorialstructurevaluedistributions} that the number of bent components of any plateaued APN function over an even dimension is $2 \mod 4$, and so $|\bentcomps{F}| \equiv 2 \mod 4$.
    Therefore, $\widehat{1_\bentcomps{F}}(c) \equiv 2 - 2|B_1| \mod 4$.
    So $\widehat{1_\bentcomps{F}}(c) \equiv 2 \mod 4$ if and only if $|B_1|$ is even.
    Since the parity of $|B_0|$ and $|B_1|$ is the same, it follows that $\mult_{\graph F}(a,b)$ is odd if and only if $|\bentcomps{F} \cap \set{0,F(a) +b}^\perp|$ is even.
    Recently, it was shown in \cite[Proposition 3.1]{beneteau2025walshspectraquadraticapn} that for any plateaued APN function with $n$ even, the intersection of $\F_2^n \setminus (\bentcomps{F} \cup \set{0})$ with any subspace of dimension at least $\frac{n}{2}+1$ is odd.
    In particular, this implies that $|\bentcomps{F} \cap \set{0,F(a)+b}^\perp|$ is even.
\end{proof}

In Section~\ref{subsec:quadratic-functions}, we provide another proof of the above proposition for quadratic APN functions via the ortho-derivative.
We now express the exclude multiplicities of the graph of a plateaued APN function in terms of $\widehat{W_F^4}$.

\begin{proposition}\label{prop:plateaued-WF4-mults}
    Let $F \colon \F_2^n\to \F_2^n$ be a plateaued APN function.
    For any $(a,b) \notin \graph F$, we have 
    \[
    \mult_{\graph F}(a,b) = \frac{1}{6 \cdot 2^{3n}}\widehat{W_F^4}(0,b+F(a)+F(0)).
    \]
    As a consequence, if $F(0)=0$ and $c \neq 0$, then $2^n \widehat{W_F^3}(0,c) = \widehat{W_F^4}(0,c)$.
\end{proposition}
\begin{proof}
    Let $(a,b) \in (\F_2^n)^2 \setminus \graph F$, and let $c = b+F(a)$.
    Notice that 
    \begin{align*}
        \frac{1}{2^{2n}} \widehat{W_F^4}(0,c)&=|\set{(x,y,z,w) \in (\F_2^n)^4 : (x+y+z+w, F(x)+F(y)+F(z)+F(w)) = (0,c)}| \\
        &=|\set{(x,y,z) \in (\F_2^n)^3 :  F(x)+F(y)+F(z)+F(x+y+z) = c}| \\
        &= |\set{(x,u,v) \in (\F_2^n)^3 :  F(x)+F(x+u)+F(x+v)+F(x+u+v) = c}| \\
        &= \sum_{x \in \F_2^n} |\set{(u,v) \in (\F_2^n)^2 : D_u D_v F(x)=c}|.
    \end{align*}
    By Theorem~\ref{thm:Carlet-PlateauedCharacterization} and eq.~(\ref{eq:second-order-deriv-excludemult}), we then know that $\frac{1}{6 \cdot 2^{3n}}\widehat{W_F^4}(0,c) =  \mult_{\graph F}(0,F(0)+c)$.
    The equation then follows as $\mult_{\graph F}(a,b) = \mult_{\graph F}(0,b+F(a)+F(0))$ by  Proposition~\ref{prop:plateaued-uniformity}.
    In the case that $F(0)=0$, then $2^n \widehat{W_F^3}(0,c) = \widehat{W_F^4}(0,c)$ for all nonzero $c\in \F_2^n$ as $\mult_{\graph F}(0,c) = \frac{1}{6\cdot 2^{2n}}\widehat{W_F^3}(0,c)$.
\end{proof}

From the above result, we can also express the exclude multiplicities of $\graph F$ in terms of the derivatives of $\gamma_F$.
In Section~\ref{sec:gammaF-linearstr}, we will see why this is relevant and how it is connected to the existence of nontrivial linear structures of $\gamma_F$.
In \cite[Section 6]{GammaCarlet}, it was shown that for any APN function $F$, we have $W_F^2 = 2^{2n}\delta_{(0,0)} - W_{\gamma_F} + 2^n$, (where $\delta_{(0,0)}$ takes value $1$ at $(0,0)$ and $0$ elsewhere).
So
\[
W_F^4 = \delta_{(0,0)} \cdot (2^{4n} -2^{2n+1}W_{\gamma_F} +2^{3n+1}) + W_{\gamma_F}^2 - 2^{n+1}W_{\gamma_F} + 2^{2n}.
\]
Since $F$ is APN, the Hamming weight $\wt(\gamma_F) = \sum_{(a,b)  \in (\F_2^n)^2}\gamma_F(a,b)$ of $\gamma_F$ is equal to $\binom{2^n}{2}$, see \cite{carletCharpinZinovievCodesBentDES}.
Therefore $W_{\gamma_F}(0,0) =2^{2n} - 2\wt(\gamma_F) =2^n$.
Applying the Fourier-Hadamard transform to both sides of the equation above, we have 
\begin{align*}
\widehat{W_F^4} 
    &= 2^{4n} - 2^{2n+1}W_{\gamma_F}(0,0) + 2^{3n+1} + \widehat{W_{\gamma_F}^2} - 2^{n+1}\widehat{W_{\gamma_F}} +2^{4n}\delta_{(0,0)}\\
    &= 2^{4n} - 2^{3n+1} + 2^{3n+1} + \widehat{W_{\gamma_F}^2} - 2^{n+1}\widehat{W_{\gamma_F}} +2^{4n}\delta_{(0,0)} \\
    &= 2^{4n} + 2^{4n}\delta_{(0,0)}  + \widehat{W_{\gamma_F}^2} - 2^{n+1}\widehat{W_{\gamma_F}}.
\end{align*}
Observe that $\widehat{W_{\gamma_F}} = 2^{2n} + 2^{3n} \delta_{(0,0)} - \widehat{W_F^2} = 2^{2n} + 2^{3n} \delta_{(0,0)} - 2^{2n}\delta_F$.
So $\widehat{W_{\gamma_F}}(0,b)=2^{2n}$ for all $b \neq 0$.
Therefore, for all $b \neq0$, we have $\widehat{W_F^4}(0,b)
    = 2^{4n} + \widehat{W_{\gamma_F}^2} - 2^{3n+1}$, so 
\begin{equation}\label{eq:WF4-quadratic-autocorrelation-gamma}
    \widehat{W_F^4}(0,b)
    =2^{3n}(2^n-2) + 2^{2n}\Delta_{\gamma_F}(0,b)
\end{equation}
as the autocorrelation $\Delta_{\gamma_F}(u,v)=\sum_{(x,y) \in (\F_2^n)^2}(-1)^{\gamma_F(x,y)+\gamma_F(x+u, y+v)}=2^{2n}-2\wt(D_{(u,v)}\gamma_F)$ of $\gamma_F$ is equal to $\frac{1}{2^{2n}}\widehat{W_{\gamma_F}^2}$ (cf. \cite{CarletBook}).
We then have the following proposition.
\begin{proposition}\label{prop:mult-autocorrelation}
    Let $F \colon \F_2^n \to \F_2^n$ be a plateaued APN function.
    Then for all $(a,b) \in (\F_2^n)^2 \setminus \graph F$, we have the equality
    \[
    \mult_{\graph F}(a,b)=  \frac{2^n-2}{6} + \frac{\Delta_{\gamma_F}(0,b+F(a))}{6 \cdot 2^n}.
    \]
\end{proposition}
\begin{proof}
    Let $(a,b) \in (\F_2^n)^2 \setminus \graph F$.
    By Proposition~\ref{prop:plateaued-uniformity}, we have $\mult_{\graph F}(a,b)=\mult_{\graph F}(0,b+F(a))$.
    Therefore, by applying Proposition~\ref{prop:plateaued-WF4-mults} and eq.~(\ref{eq:WF4-quadratic-autocorrelation-gamma}), we have
\begin{align*}
    \mult_{\graph F}(a,b) &=\mult_{\graph F}(0,b+F(0)+F(a))  \\
    &= \frac{1}{6 \cdot 2^{3n}}\widehat{W_F^4}(0,b+F(a)) \\
    &= \frac{2^n-2}{6} + \frac{\Delta_{\gamma_F}(0,b+F(a))}{6 \cdot 2^n}.
\end{align*}
\end{proof}

\subsection{Plateaued $3$-to-$1$ functions}\label{subsec:3to1plat-functions}

Throughout this subsection, we assume that $n$ is even.
We say a function $F \colon \F_2^n \to \F_2^n$ is $3$-to-$1$ if every point in the image of $F$ has a preimage of size $3$ except a single point.
Note that all plateaued $3$-to-$1$ functions are necessarily APN (see \cite[Corollary 8]{KolschImageSets2023}).

A subset $D \subseteq G$ of an additive group $G$ is a \textit{partial difference set} with parameters $(v,k,\lambda, \mu)$ if every non-identity element of $D$ (resp. $G \setminus D$) can be written as $x-y$ with distinct $x,y \in D$ in exactly $\lambda$ (resp. $\mu$) ways.
A function $F \colon \F_2^n \to \F_2^n$ is called \textit{crooked} if for any nonzero $a \in \F_2^n$, the image of the derivative $D_a F$ is an affine hyperplane (note that this implies $F$ is APN).
The following theorem was originally only proved for crooked $3$-to-$1$ functions, but its proof only relied on the fact that all crooked functions are plateaued (see \cite{CarletBook}), so we write it for all plateaued $3$-to-$1$ functions.
\begin{theorem}[{\cite[Theorem 4]{BudaghyanTriplicateFunctions}}]\label{thm:plateaued-PDS}
    Assume $n$ is even, and let $F \colon \F_2^n \to \F_2^n$ be a plateaued function such that $F(0)=0$ and $F$ is $3$-to-$1$ on $\F_2^n \setminus \set{0}$.
    Then $\im (F) \setminus \set{0}$ is a partial difference set, with parameters $(2^n, \frac{2^n-1}{3}, \frac{2}{3}(\alpha(n)-1), \frac{2}{3}\beta(n))$, where $\alpha(n) = \frac{2^n + (-2)^{\frac{n}{2}+1}- 2}{6}$ and $\beta(n) = \frac{2^n+(-2)^{\frac{n}{2}} - 2}{6}$.
\end{theorem}

Moreover, \cite[Corollary 4]{BudaghyanTriplicateFunctions} describes the multiplicities of the nonzero elements of the multiset $\set{F(x) + F(y) +F(x+y) : x,y \in \F_2^n}$ when $F$ is quadratic $3$-to-$1$, and indeed, this can be easily generalized to all plateaued $3$-to-$1$ APN functions.

\begin{corollary}\label{cor:PDS-mults}
    Assume $n$ is even, let $\alpha(n)$ and $\beta(n)$ be defined as in Theorem~\ref{thm:plateaued-PDS}, and let $F \colon \F_2^n \to \F_2^n$ be a plateaued $3$-to-$1$ function with $F(0)=0$ such that $\im(F)\setminus \set{0}$ is a partial difference set with parameters $(2^n, \frac{2^n-1}{3}, \frac{2}{3}(\alpha(n)-1), \frac{2}{3}\beta(n))$.
    Let $b \in \F_2^n \setminus \set{0}$.
    Then $\mult_{\graph F}(0,b) = \alpha(n)$ if $b\in \im (F)$ and $\mult_{\graph F}(0,b) = \beta(n)$ if $b \notin \im (F)$.
\end{corollary}

These two results above (in terms of quadratic $3$-to-$1$ functions) and \cite[Proposition 5]{DistanceBetweenAPNFunctions2020} (which states that $\Pi_F^\beta$ does not depend on $\beta \in \F_2^n$ when $F$ is quadratic) were used in \cite{BudaghyanTriplicateFunctions} to determine the exact exclude multiplicities of $\graph F$ and their frequencies in the case of $F$ being a quadratic $3$-to-$1$ function. 
This can be done because any plateaued function $F$ with all component functions unbalanced satisfies 
\begin{equation}\label{eq:second-order-deriv-image-differences}
    |\set{(a,b) \in (\F_2^n)^2 : D_a D_b F(x) = w}| = |\set{(a,b) \in (\F_2^n)^2 : F(a)+F(b) = w}|
\end{equation}
for all $x,w \in \F_2^n$ (see \cite[Theorem 2]{CarletPlateaued}).
Therefore, if $F$ is plateaued APN with all unbalanced components and $(a,b) \in (\F_2^n)^2 \setminus \graph F$, then 
\[
\mult_{\graph F}(a,b)= \frac{1}{6} |\set{(u,v) \in (\F_2^n)^2 : F(u)+F(v)=F(a)+b}|
\]
by combining eqs.~(\ref{eq:second-order-deriv-excludemult},\ref{eq:second-order-deriv-image-differences}).
In Proposition~\ref{prop:plateaued-uniformity}, we proved that the exclude multiplicities of the graph of a plateaued APN function $F$ are completely determined by the exclude multiplicities of points in $\set{0} \times (\F_2^n \setminus \set{0})$.
We use this to prove the following result.

\begin{theorem}\label{thm:plateaued-3to1-mults}
    Assume $n\geq 4$ is even, let $\alpha(n)$ and $\beta(n)$ be defined as in Theorem~\ref{thm:plateaued-PDS}, and let $F \colon \F_2^n \to \F_2^n$ be a plateaued $3$-to-$1$ function.
    Then 
    \begin{enumerate}
        \item there are $2^n \cdot \frac{2^n -1}{3}$ exclude points of $\graph F$ with multiplicity $\alpha(n)$;
        \item there are $2^{n+1} \cdot \frac{2^n -1}{3}$ exclude points of $\graph F$ with multiplicity $\beta(n)$;
        \item every point in $(\F_2^n)^2 \setminus \graph F$ has exclude multiplicity $\alpha(n)$ or $\beta(n)$, with respect to $\graph F$.
    \end{enumerate}
\end{theorem}
\begin{proof}
    Without loss of generality, assume $F(0)=0$.
    By Theorem~\ref{thm:plateaued-PDS}, $\im (F) \setminus \set{0}$ is a partial difference set with parameters $(2^n, \frac{2^n-1}{3}, \frac{2}{3}(\alpha(n)-1), \frac{2}{3}\beta(n))$.
    The result then immediately follows by applying Proposition~\ref{prop:plateaued-uniformity} and Corollary~\ref{cor:PDS-mults} (cf. \cite{BudaghyanTriplicateFunctions}) and using the same proof as \cite[Corollary 4]{BudaghyanTriplicateFunctions}.
\end{proof}

We now prove Theorem~\ref{thm:distance-to-plat3to1}.

\begin{proof}[Proof of Theorem~\ref{thm:distance-to-plat3to1}]
 For any plateaued $3$-to-$1$ function $F \colon \F_2^n \to \F_2^n$ and any point $(a,b) \in (\F_2^n)^2 \setminus \graph {F}$, we have that $\mult_{\graph F}(a,b) \in \set{\alpha(n), \beta(n)}$, where $\alpha(n)$ and $\beta(n)$ are defined as above.
 Observe that $\beta(n) > \alpha(n)$ when $n \equiv 0 \mod 4$, and $\alpha(n) > \beta(n)$ when $n \equiv 2 \mod 4$.
 The statement follows immediately from applying Lemma~\ref{lem:MinimumDistance-By-MinimumExcludeMult} in the same way as in the proof of Theorem~\ref{thm:plat-dist-bound}.
\end{proof}

In this subsection, we have generalized the known lower bounds on the Hamming distance from a quadratic $3$-to-$1$ function to any other APN function.
In particular, our generalization provides lower bounds on the distance from plateaued $3$-to-$1$ functions to any other APN functions, and plateaued $3$-to-$1$ functions include Kasami functions as shown by Yoshiara \cite{YoshiaraKasamiPlateaued}.

\subsection{Quadratic APN functions}\label{subsec:quadratic-functions}

By definition, a vectorial Boolean function $F \colon \F_2^n \to \F_2^n$ is quadratic if every derivative $D_a F(x)=F(x)+F(x+a)$ of $F$ is an affine function.
Quadratic APN functions are particularly well-behaved as the images of their derivatives are affine hyperplanes (more generally, this property is called crookedness, see \cite{CarletBook, KyureghyanCrooked2007}).
Moreover, quadratic APN functions are equipped with a unique so-called ortho-derivative.

\begin{definition}
    Let $F \colon \F_2^n \to \F_2^n$ be a quadratic APN function.
    The \textit{ortho-derivative} of $F$ is the unique function $\pi_F \colon \F_2^n \to \F_2^n$ such that $\pi_F(0) =0$ and $\pi_F(a)\neq0 $ for all $a \in \F_2^n \setminus \set{0}$ and 
    \[
    \pi_F(a) \cdot (F(x)+F(x+a)+F(0)+F(a)) = 0
    \]
    for all $a,x \in \F_2^n$.
\end{definition}

The ortho-derivative of a quadratic APN function $F$ is the function $\pi_F$ such that $\pi_F(0)=0$ and $\set{0,\pi_F(a)}^\perp$ is the underlying vector space of the affine hyperplane $\im( D_a F)$ for all nonzero $a \in \F_2^n$ (the concept of the ortho-derivative has been known since \cite{carletCharpinZinovievCodesBentDES} but was not given a name until recently in \cite{CouvreurOrtho-Derivative}).
In \cite{CouvreurOrtho-Derivative}, properties of the ortho-derivative were studied, and in particular, there exists a connection between the ortho-derivative of $F$ and the exclude multiplicities of $\graph F$.

\begin{proposition}[\cite{CouvreurOrtho-Derivative}]
    Let $F \colon \F_2^n \to\F_2^n$ be a quadratic APN function, and let $\pi_F \colon\F_2^n \to \F_2^n$ be its ortho-derivative.
    Then for any $b \in \F_2^n \setminus \set{0}$, 
    \[
        \frac{1}{2^{2n}} \widehat{W_F^4}(0,b) = 2^{n+1}(2^n - 1 -\wt(b \cdot \pi_F)).
    \]
\end{proposition}

Therefore if $F$ is a quadratic APN function with $F(0) =0$ and $b \neq 0$, then
\begin{equation}\label{eq:mult-orthoderiv-components}
\mult_{\graph F}(0,b) =\frac{1}{6 \cdot 2^{2n}}\widehat{W_F^3}(0,b)  = \frac{1}{6 \cdot 2^{3n}}  \widehat{W_F^4}(0,b) = \frac{1}{3}(2^n - 1 -\wt(b \cdot \pi_F))
\end{equation}
with the second equality following from Proposition~\ref{prop:plateaued-WF4-mults}.
\begin{remark}
    Let $F$ be a quadratic APN function with $F(0)=0$.
    If $n$ is odd, then all nonzero components $b \cdot \pi_F$ of the ortho-derivative of $F$ are balanced, i.e. $\pi_F$ is a permutation (this was observed in \cite{CouvreurOrtho-Derivative}).
    This is because when $n$ is odd, $F$ is AB implying $\wt( b\cdot \pi_F) = 2^n-1-3 \cdot \frac{2^{n-1}-1}{3}=2^{n-1}$ by Theorem~\ref{thm:AB-vanDamFlaass}.
    If $n$ is even, then $\wt(b \cdot \pi_F)$ is a multiple of $3$ for all $b \neq 0$ as $3$ divides $2^n-1-\wt(b \cdot \pi_F)$ and $2^n \equiv 1 \mod 3$.
\end{remark}
Note that eq.~(\ref{eq:mult-orthoderiv-components}) shows that for a quadratic APN function $F$, obtaining a lower bound on the exclude points of $\graph F$ is equivalent to finding an upper bound on the weights of the components of $\pi_F$.
By applying Theorem~\ref{thm:plat-exclude-bound}, we then obtain an upper bound on the weights of the component functions of $\pi_F$.
\begin{proposition}
    Assume $n$ is even.
    Let $F \colon \F_2^n \to \F_2^n$ be a quadratic APN function with $F(0)=0$.
    For any $b \neq 0$, the weight of the component function $b \cdot \pi_F$ of the ortho-derivative of $F$ satisfies
    $\wt(b \cdot \pi_F) \leq 2^n +2 -3 \cdot 2^{\frac{n}{2}-1}$.
\end{proposition}
\begin{proof}
    By Theorem~\ref{thm:plat-exclude-bound}, we have that $\mult_{\graph F}(0,b) \geq2^{\frac{n}{2}-1}-1$ for all $b \neq 0$.
    The result then follows by applying this bound to eq.~(\ref{eq:mult-orthoderiv-components}).
\end{proof}

It was also shown in \cite{CouvreurOrtho-Derivative} that the ortho-derivative of a quadratic APN function has algebraic degree at most $n-2$.
For any function $F \colon \F_2^n \to \F_2^n$, it is well-known that $F$ has algebraic degree $n$ if and only if $\sum_{x \in \F_2^n} F(x) \neq 0$.
Therefore, if $F \colon \F_2^n \to \F_2^n$ is quadratic APN, then $\sum_{x \in \F_2^n} \pi_F(x) =0$, implying that all component functions of $\pi_F$ have even Hamming weight as for any $v \neq 0$, we have
$
\wt(v \cdot \pi_F) \equiv \sum_{x \in \F_2^n} v \cdot \pi_F(x) \pmod 2 = v \cdot \sum_{x \in \F_2^n} \pi_F(x) = 0.
$
This allows us to prove that the exclude multiplicities of the graph of a quadratic APN function are always odd in a different manner than Proposition~\ref{prop:plateaued-odd-mults}.

\begin{proposition}\label{prop:quadratic-odd-mults}
    Assume $n \geq 3$.
    Let $F \colon \F_2^n \to \F_2^n$ be a quadratic APN function.
    Then $\mult_{\graph F}(a,b)$ is odd for all $(a,b) \in (\F_2^n)^2 \setminus \graph F$.
    Equivalently, $|\mathcal{B}(F) \cap H|$ is even for any affine hyperplane $H$ of $\F_2^n$.
\end{proposition}
\begin{proof}
    If $n$ is odd, then $F$ is AB and $\mult_{\graph F}(a,b) = \frac{2^n-2}{6}$ for all $(a,b) \notin \graph F$.
    So, assume $n$ is even, and without loss of generality, assume $F(0)=0$.
    As mentioned above, the algebraic degree of $\pi_F$ is at most $n-2$, implying every component function of $\pi_F$ has even Hamming weight.
    Therefore, by eq.~(\ref{eq:mult-orthoderiv-components}), for any $b \neq 0$, we know that $\mult_{\graph F}(0,b) = \frac{1}{3}(2^n-1-\wt(b \cdot \pi_F))$ is odd.
    Thus, by applying Proposition~\ref{prop:plateaued-uniformity}, we know that any exclude point of $\graph F$ has odd multiplicity.
\end{proof}

Note that it does not always hold that the exclude multiplicities of the graph of an APN function must be odd.
For example, when $n =5$, the graph of the APN inverse function has exclude points of multiplicity $3,4,5$, and $6$.

As a corollary of Proposition~\ref{prop:quadratic-odd-mults}, every quadratic APN function has a maximal Sidon set as its graph, and as previously mentioned, this is already known since all quadratic functions are plateaued.
Nonetheless, it is interesting that Proposition~\ref{prop:quadratic-odd-mults} now provides a fourth method for proving no quadratic APN function (for $n \geq 3$) has a non-maximal Sidon set as its graph (the other three methods being the work of \cite{budaghyanCarletHellesetUpperBoundsDegree}, the usage of the D-property in \cite{carlet_apnGraphMaximal}, and Theorem~\ref{thm:plat-exclude-bound}).

\section{Power functions}\label{sec:power-fcns}

In this section, we consider the case when $F \colon \F_{2^n} \to \F_{2^n}$ is of the form $F(x)=x^d$ for some integer $d$.
If $F$ is an APN power function, then we will see that the possible values of $\mult_{\graph F}(a,b)$ do not depend on $a \in \F_{2^n}^\ast$ as $b$ ranges across $\F_{2^n}$ (i.e. $\Pi_F^\beta$ does not depend on $\beta \in \F_{2^n}^\ast$).
Also, we will fully determine $\mult_{\graph F}(0,a)$ and $\mult_{\graph F}(a,0)$ when $F$ is APN, $n$ is odd, and $a \neq 0$.

For any power function $F(x)=x^d$, we have $W_F(u,v) = W_F(1,u^{-d}v)$ for all $u \neq 0$, and in the case when $F$ is a permutation, we then have for $v \neq 0$ that $W_F(u,v)=W_F(uv^{-\frac{1}{d}},1)$ where $\frac{1}{d}$ is the inverse of $d$ modulo $2^n-1$ \cite{CarletBook}.
The first of these two properties allows us to prove the following result. 

\begin{proposition}\label{prop:power-functions-uniform}
    Let $F \colon \F_{2^n} \to \F_{2^n}$ be an APN power function $F(x)=x^d$.
    If $a,b,c \in \F_{2^n}$ such that $a,c \neq 0$ and $b \neq a^d$, then $\mult_{\graph F}(a,b) = \mult_{\graph F}(c,(a^{-1}c)^d b)$. 
    In particular, the exclude multiplicities of the points in $\set{a} \times (\F_{2^n} \setminus \set{a^d})$, with respect to $\graph F$, do not depend on the choice of $a \in \F_{2^n}^\ast$.
\end{proposition}
\begin{proof}
    Let $a,b \in \F_{2^n}$ such that $a \neq 0$ and $b \neq a^d$.
    Recall that we have the equality $W_F(u,v) = W_F(1, u^{-d}v)$ for all $u,v \in \F_2^n$ where $u \neq 0$.
    So
    \begin{align*}
       \widehat{W_F^3}(a,b)
        &= \sum_{(u,v) \in \F_{2^n}^2} (-1)^{\Tr(ua) + \Tr(vb)} W_F^3(u,v) \\
        &= \sum_{\substack{(u,v) \in \F_{2^n}^2 \\ u \neq 0}} (-1)^{\Tr(u) + \Tr(vb)} W_F^3(ua^{-1},v) + \sum_{v \in \F_{2^n}} (-1)^{\Tr(vb)}W_F^3(0,v)\\
         &= \sum_{\substack{(u,v) \in \F_{2^n}^2 \\ u \neq 0}} (-1)^{\Tr(u) + \Tr(vb)} W_F^3(1,u^{-d}a^d v) + \sum_{v \in \F_{2^n}} (-1)^{\Tr(vb)}W_F^3(0,v)\\
        &=  \sum_{\substack{(u,v) \in \F_{2^n}^2 \\ u \neq 0}} (-1)^{\Tr(u) + \Tr(va^{-d}b)} W_F^3(1,u^{-d} v) + \sum_{v \in \F_{2^n}} (-1)^{\Tr(vb)}W_F^3(0,v)\\
        &= \sum_{\substack{(u,v) \in \F_{2^n}^2 \\ u \neq 0}} (-1)^{\Tr(u) + \Tr(va^{-d}b)} W_F^3(u,v) + \sum_{v \in \F_{2^n}} (-1)^{\Tr(vb)}W_F^3(0,v).
    \end{align*}
    We then have that 
    \[
    \widehat{W_F^3}(a,b)
    = \widehat{W_F^3}(1,a^{-d}b) - \sum_{v \in \F_{2^n}} (-1)^{\Tr(va^{-d}b)}W_F^3(0,v) +\sum_{v \in \F_{2^n}} (-1)^{\Tr(vb)}W_F^3(0,v).
    \]
    However, notice that the following is equal to $0$
    \begin{align*}
        &- \sum_{v \in \F_{2^n}} (-1)^{\Tr(va^{-d}b)}W_F^3(0,v) +\sum_{v \in \F_{2^n}} (-1)^{\Tr(vb)}W_F^3(0,v) \\
        &= \sum_{v \in \F_{2^n}} (-1)^{\Tr(vb)}
        \parens{W_F^3(0,v) - W_F^3(0,a^dv)}
    \end{align*}
    by the equality $W_F(0,a^dv) = W_F(0,v)$.
    Therefore $\widehat{W_F^3}(a,b) = \widehat{W_F^3}(1,a^{-d}b)$, implying $\mult_{\graph F}(a,b) = \mult_{\graph F}(1,a^{-d}b)$ as $\mult_{\graph F}(x,y) = \frac{1}{6 \cdot 2^{2n}} \widehat{W_F^3}(x,y)$ for all $(x,y) \notin \graph F$ by eq.~(\ref{eq:WFk-set-size}).
    Note that $b \mapsto a^{-d}b$ is a permutation, and so the exclude multiplicities of the points in $\set{a} \times (\F_{2^n} \setminus \set{F(a)}$, with respect to $\graph F$, are exactly those that are contained in $\set{a} \times (\F_{2^n} \setminus \set{1})$.
    Finally, we can apply our equality twice to see that for any $c \neq 0$, we have $\mult_{\graph F}(a,b) = \mult_{\graph F}(1,a^{-d}b)=\mult_{\graph F}(c,(a^{-1}c)^d b)$.
\end{proof}

If $F(x)=x^d$ is a APN and $n$ is odd, we can determine the exact exclude multiplicities of points in $\set{0} \times \F_{2^n}^\ast$ with respect to $\graph F$, but first, we describe the sum $\sum_{x \in \F_{2^n}}\gamma_F(x + a, x^d + b)$ in the following way to make our computation straightforward.

\begin{proposition}
\label{prop:mult-intersection-with-imD1F}
    Let $F \colon \F_{2^n} \to \F_{2^n}$ be a power function $F(x)=x^d$, and let $a,b \in \F_{2^n}$.
    Let $\Phi \colon \F_{2^n} \setminus \set{a} \to \F_{2^n}$ be the function $\Phi(x) = \frac{x^d+b}{(x+a)^d}$.
    Then
    \[
        \sum_{y \in \F_{2^n}}\gamma_F(y+a, y^d+b) = |\Phi^{-1}(\im( D_1F)))|.
    \]
\end{proposition}
\begin{proof}
    When $y \neq a$, we have $D_{y+a}F(x) = (y+a)^dD_1F\parens{\frac{x}{y+a}}$.
    Therefore if $y \neq a$, then $\gamma_F(y+a, y^d+b)=1$ if and only if $D_1F\parens{\frac{x}{y+a}}=\frac{y^d+b}{(y+a)^d}$ has a solution.
    For any fixed $y \neq a$, the function $x \mapsto \frac{x}{y+a}$ is a permutation.
    Therefore the sum $\sum_{y \in \F_{2^n}}\gamma_F(y+a, y^d+b)$ is the number of $y\neq a$ such that $\frac{y^d+b}{(y+a)^d}$ is contained in $\im (D_1F)$, that is, the size of the preimage of $\im (D_1F)$ under $\Phi$.
    Hence
    \[
    \sum_{y \in \F_{2^n}}\gamma_F(y+a, y^d+b) = |\Phi^{-1}(\im (D_1F))|,
    \]
    as desired.
\end{proof}

With this fact and our previous results, we determine the exclude multiplicities, with respect to $\graph F$, of points in $\set{0} \times \F_{2^n}^\ast$ when $n$ is odd.
This is because in Proposition~\ref{prop:PiF-beta-excludemults}, we proved that $3\mult_{\graph F}(a,b) = \sum_{x \in \F_2^n} \gamma_F(x+a, F(x)+b)$ for all $(a,b) \in (\F_2^n)^2 \setminus \graph F$.

\begin{proposition}\label{prop:APN-power-mults-0a-a0}
    Assume $n \geq 3$.
    Let $F \colon \F_{2^n} \to \F_{2^n}$ be an APN power function $F(x)=x^d$.
    If $n$ is odd, then for all $a \neq 0$, we have 
    \[
    \mult_{\graph F}(a,0)=\mult_{\graph F}(0,a)=\frac{2^{n-1}-1}{3}.
    \]
\end{proposition}
\begin{proof}
    Assume $n$ is odd.
    For $(a,b) \in \F_{2^n}^2$, let $\Phi \colon \F_{2^n} \setminus \set{a} \to \F_{2^n}$ be defined by $\Phi(x)=\frac{x^d+b}{(x+a)^d}$.
    First, suppose $a \neq 0$ and $b =0$.
    Then $\Phi(x) = \frac{x^d}{(x+a)^d} = \parens{\frac{x}{x+a}}^d$.
    For $x \neq a$, we have $\frac{x}{x+a} =a(x+a)^{-1}+1$. 
    Note $x \mapsto a(x+a)^{-1}+1$ is a permutation from $\F_{2^n} \setminus \set{a}$ is a permutation onto its image which is $\F_{2^n} \setminus \set{1}$.
    So $|\Phi^{-1}(\im (D_1 F))| = |\set{x\in \F_{2^n} \setminus \set{1} : x^d \in \im (D_1 F)}|$.
    Note that $F(x)=x^d$ is a permutation since $n$ is odd.
    Hence $1 \notin \im (\Phi)$ but $1 \in \im (D_1F)$.
    Applying Proposition~\ref{prop:PiF-beta-excludemults} and Proposition~\ref{prop:mult-intersection-with-imD1F}, we have 
    \[
    \mult_{\graph F}(a,0) =\frac{|\Phi^{-1}(\im (D_1F))|}{3} =\frac{|\im (D_1 F)| - 1}{3} = \frac{2^{n-1}-1}{3}.
    \]

    Now, suppose $a = 0$ and $b \neq 0$.
    Then $\Phi(x) = \frac{x^d +b}{x^d}=bx^{-d}+1$.
    Hence $\Phi$ is a permutation onto its image which is $\F_{2^n} \setminus \set{1}$, and similar to before we have $1 \notin \im (\Phi)$ but $1 \in \im (D_1F)$.
    Therefore
    \[
    \mult_{\graph F}(0,a) = \frac{|\Phi^{-1}(\im (D_1F))|}{3} =\frac{2^{n-1}-1}{3}.
    \]
\end{proof}

In \cite[Conjecture 21]{Kaleyski2019TwoPoints}, Kaleyski conjectured that if $F$ is an APN power function, the two values $\sum_{y \in \F_{2^n}} \delta_F(y, F(y)+1)$ and $\sum_{y \in \F_{2^n}} \delta_F(y+1, F(y))$ are equal to $\sum_{y \in \F_{2^n}}\delta_{x^3}(y, y^3 +1)$.
By Proposition~\ref{prop:PiF-beta-excludemults}, this is equivalent to $\mult_{\graph F}(0,1) = \mult_{\graph F}(1,0) = \mult_{\graph {x^3}}(0,1)$.
Therefore, Proposition~\ref{prop:APN-power-mults-0a-a0} proves Kaleyski's conjecture in the case of $n$ odd, but the case of $n$ even still remains open.

\begin{remark}
    The reader may have noticed that the particular case of $\Phi(x) = \frac{x^d+b}{(x+a)^d}$ when $a=b=1$ was studied by Carlet and Picek in \cite{CarletPicek}.
    In particular, they proved $x \mapsto x^d$ is APN if and only if $x \mapsto \frac{x^d+1}{(x+1)^d}$ is $2$-to-$1$ from $\F_{2^n}\setminus \F_2 \to \F_{2^n} \setminus \set{1}$.
    More generally, it is true that $F(x)=x^d$ is APN if and only if for all $a \in \F_{2^n}^\ast$ the map $\Phi(x)= \frac{x^d+a^d}{(x+a)^d}$ is $2$-to-$1$ from $\F_{2^n}\setminus \set{0,a}$ to $\F_{2^n} \setminus \set{1}$.
    To see this, suppose $x \neq a$ and let $s = \frac{x}{x+a}$.
Observe that $x \mapsto s$ is a permutation from $\F_{2^n} \setminus \set{a}$ to $\F_{2^n}\setminus \set{1}$.
Also $D_1F(s)=s^d + (s+1)^d = \frac{x^d+a^d}{(x+a)^d}$.
Then $x \mapsto \frac{x^d + a^d}{(x+a)^d}$ is $2$-to-$1$ from $\F_{2^n}\setminus \set{0,a}$ to $\F_{2^n} \setminus \set{a}$
if and only if for any $b \neq 1$, the equation $D_1F(s)=b$ has at most $2$ solutions for $s \in \F_{2^n}$ (and there are no solutions in $\set{0,a}$).
Therefore $x \mapsto \frac{x^d + a^d}{(x+a)^d}$ is $2$-to-$1$ from $\F_{2^n}\setminus \set{0,a}$ to $\F_{2^n} \setminus \set{a}$ if and only if $D_1F$ is $2$-to-$1$, which is equivalent to $F$ being APN.
\end{remark}

\section{Lower bounds on distances from the APN inverse function}\label{sec:inverse}

Throughout this section, we assume that $n$ is odd and define $F \colon \F_{2^n} \to \F_{2^n}$ to be the multiplicative inverse function $F(x)=x^{-1}$, where $\frac{1}{0}:=0$.
To determine lower bounds on the Hamming distance from $F(x)=x^{-1}$ over $\F_{2^n}$ to any other APN function, we will determine lower bounds on the multiplicities of the exclude points of $\graph F$.
In particular, we derive a lower bound in terms of the binary Kloosterman sums
\[
K_n(a) = \sum_{x \in \F_{2^n}} (-1)^{\Tr(ax + x^{-1})}.
\]
For notational convenience, we also denote by $K_n^\ast(a)$ the sum 
\[
K_n^\ast(a)  = \sum_{x \in \F_{2^n}^\ast} (-1)^{\Tr(ax + x^{-1})}.
\]
A well-known result is that $\set{K_n^\ast(a) : a \neq 0}$ is the set of integers in $[-2^{\frac{n}{2}+1}, 2^{\frac{n}{2}+1}]$ that are congruent to $3 \mod 4$, see \cite{LachaudWolfmannGoppa}.
Kloosterman sums and the Walsh transform of $F$ have a very close connection as one can easily verify that $W_F(u,v)=K_n(uv) + 2^n \delta_{(0,0)}(u,v)$.

Kloosterman sums can also be used to describe the number of rational points on ordinary elliptic curves over $\F_{2^n}$.
An \textit{algebraic plane curve} $\mathcal{C}$ over $\F_{2^n}$ is defined by the equation $p(x,y) = 0$ for some irreducible polynomial $p(x,y)$ over $\F_{2^n}$.
If $\widetilde{p}(x,y,z)$ is the associated homogenized polynomial to $p(x,y)$, we denote the solution set to $\widetilde{p}(x,y,z)=0$ as $\widetilde{\mathcal{C}}$.
We say that the number of \textit{rational points} of $\mathcal{C}$ is the size of $\widetilde{\mathcal{C}}$.
Also, a point $(x,y,z) \in \widetilde{\mathcal {C}}$ is \textit{singular} if all partial derivatives of $\widetilde{p}(x,y,z)$ vanish at $(x,y,z)$.
If there are no singular points, we define the \textit{genus} of $\mathcal{C}$ as $g = \frac{(d-1)(d-2)}{2}$, where $d$ is the degree of $\widetilde{p}(x,y,z)$.
If the genus of $\mathcal{C}$ is $1$, then we say that $\mathcal{C}$ is \textit{elliptic}.
Moreover, we say that $\mathcal{C}$ is \textit{supersingular} if the coefficient of $xyz$ is zero in $\widetilde{p}(x,y,z)$, otherwise we say that $\mathcal{C}$ is \textit{ordinary} (this is not the typical definition of supersingular and ordinary elliptic curves, but it is equivalent due to a result of Hartshorne, see \cite[Proposition 4.21]{Hartshorne}).

Lachaud and Wolfmann proved the following two results in \cite{LachaudWolfmannGoppa}.

\begin{proposition}[\textup{\cite{LachaudWolfmannGoppa}[Corollary 2.2]}]
    \label{prop:ordinary-elliptic-isomorphic}
    An ordinary elliptic curve $\mathcal{E}$ defined over $\F_{2^n}$ is isomorphic to one of the following Kloosterman curves 
    \begin{align*}
    \mathcal{KL}^+_a &\colon y^2 +y = ax + x^{-1} \quad a \in \F_{2^n}^\ast, \\
    \mathcal{KL}^-_a &\colon y^2 +y = ax + x^{-1} +\tau \quad a,\tau \in \F_{2^n}, a \neq 0.
    \end{align*}
\end{proposition}

\begin{proposition}[\textup{\cite{LachaudWolfmannGoppa}[Corollary 3.3]}]\label{prop:LachaudWolfmann-rational-points}
    For any $a \neq 0$, the number of rational points of $\mathcal{KL}^\pm_a$ is given by $2^n +1 \pm K_n^\ast(a)$.
\end{proposition}

We will now prove the following theorem and then use Lemma~\ref{lem:MinimumDistance-By-MinimumExcludeMult} to obtain lower bounds on the distance of any other APN function to $F$.

\begin{theorem}\label{thm:inverse-minmult}
    Assume $n \geq 3$ is odd.
    Let $F \colon \F_{2^n} \to \F_{2^n}$ be the function $F(x) =x^{-1}$ where we define $F(0)=0$.
    Then for all $(a,b) \in \F_{2^n}^2 \setminus \graph F$, we have 
    \begin{equation}\label{eq:inverse-minmult}
    \mult_{\graph F}(a,b) \geq \ceil{\frac{2^n-5 - \max_{v \in \F_{2^n}^\ast} |K_n(v)-1|}{6}}.
    \end{equation}
\end{theorem}
\begin{proof}
Let $(a,b) \notin \graph{F}.$ The exclude multiplicity of $(a,b)$ is $\frac{s}{6}$ where $s$ counts the number of solutions $(x,y,z) \in \F_{2^n}^3$ to
\[
\begin{cases}
    x+y+z=a\\
    x^{-1}+y^{-1}+z^{-1}=b,
\end{cases}
\]
Equivalently, $s$ is the number of roots $(x,y) \in$ $\F_{2^n}^2$ of the polynomial
\begin{equation}\label{eq:inverse-mult-eqn}
    p(x,y) = x^{-1} + y^{-1} + (x + y + a)^{-1}+ b.
\end{equation}
By Proposition~\ref{prop:APN-power-mults-0a-a0}, we know that if either $a$ or $b$ is zero, then $s=2^n-2$ since $F$ is APN and $n$ is odd. 
Moreover, as shown in Proposition~\ref{prop:power-functions-uniform}, the value of $s$ is independent of $a$ as $a$ ranges across $\F_{2^n}^\ast$.
So, without loss of generality, assume $a = 1$ and $b \notin \F_2$.
Define the affine plane curve $\mathcal C = \set{(x,y) \in \F_{2^n}^2 : p(x,y) =0},$ and let $D \subset \mathcal{C}$ be the subset where no denominator vanishes: 
\[D =\set{(x,y) \in \mathcal{C}: 0 \notin \set{x,y,x+y+1}}\]
We claim that $\#\mathcal{C} = |D| + 3\delta_F(1,b)$. 
Indeed, if $x = 0$, then $p(0,y) = y^{-1} + (y + 1)^{-1} + b = 0$ has $\delta_F(1,b)$ solutions. 
Similarly, the cases $y = 0$ and $x + y = 1$ each contribute $\delta_F(1,b)$ distinct solutions.
Hence, $\#\mathcal{C} = |D| + 3\delta_F(1,b)$.

Now we multiply eq.~(\ref{eq:inverse-mult-eqn}) by $xy(x + y + 1)$ to obtain the following polynomial in $\F_{2^n}[x,y]$:
\begin{equation}
    q(x,y) = bx^2y + bxy^2 + x^2 + y^2 + bxy + xy + x + y.
\end{equation}
Let $\mathcal{E} = \{(x,y) \in \F_{2^n}^2: q(x,y) = 0\}$, and let $G=\set{(x,y)\in \mathcal{E} : 0 \notin \set{x,y,x+y+1}}$.
Note that for all $(x,y) \in \F_{2^n}^2$ such that $x,y, x+y + 1 \neq 0$, we have $p(x,y) = 0$ if and only if $q(x,y) = 0.$ 
Hence $|D|=|G|$.

We claim that $\# \mathcal{E} = |G| + 3$.
To see this, let us consider a point $(x,y) \in \mathcal{E} \setminus G$.
If $x = 0,$ then $0 = q(0,y) = y^2 + y.$ Hence $y \in \F_2$. Similarly, if $y = 0,$ then $x\in \F_2$.
Finally, if $x + y = 1,$ then 
\[0 = bx^2(x + 1) + bx(x+1)^2 + bx(x+1) + x(x + 1) + x^2 + (x+1)^2 + x + (x+1) = x^2 + x.\]
This implies $x \in \F_2$, and similarly $y \in \F_2$.
Since $q(1,1) = b +1 \neq 0$, we have $(x,y) \in \F_2^2 \setminus \set{(1,1)}$.
Therefore $\# \mathcal{E} = |G| + 3$.
 
Now, we homogenize  $q(x,y)$ to obtain the homogeneous polynomial
\[
\tilde{q}(x,y,z) = bx^2y + bxy^2 + x^2z + y^2z + (b + 1)xyz + xz^2 + yz^2,
\]
Let $\widetilde{\mathcal{E}} = \{(x,y,z) : \tilde{q}(x,y,z) = 0\}$ denote the projective curve defined by $\tilde{q}.$
By considering the partial derivatives of $\tilde{q}$, it is easy to see that $\widetilde{\mathcal{E}}$ is smooth, and we know that $\mathcal{E}$ is elliptic because its genus is $g = \frac{(3-1)(3-2)}{2}=1$.
Moreover, $\mathcal{E}$ is ordinary because $xyz$ has a nonzero coefficient in $\tilde{q}$.
Note that the projective points of $\widetilde{\mathcal{E}}$ include three additional points, corresponding to solutions of the form $(x:y:0)$ that satisfy $\widetilde{q}(x,y,0) = bx^2y + bxy^2 = bxy(x + y) =0$. 
There are three such cases: $x = 0, y = 0$ and $x+y = 0,$ corresponding to the points $(1:0:0), (0:1:0), (1:1:0) \in \widetilde{\mathcal{E}}$ respectively. 
So $\# \widetilde{\mathcal{E}} = \# \mathcal{E} +3$.

Hence, $6 \mult_{\graph F}(a,b) = s = \# \mathcal{C} = |D| + 3 \delta_F(1,b) = |G| + 3\delta_F(1,b)=\#\mathcal{E} -3 +3\delta_F(1,b)$, implying $6\mult_{\graph F}(a,b) 
    = \#\widetilde{\mathcal{E}} - 6 + 3\delta_F(1,b)$.
By Proposition~\ref{prop:LachaudWolfmann-rational-points}, we have that $\# \widetilde{\mathcal{E}}$ is equal to $2^n +1 \pm K_n^\ast(u)$ for some $u \in \F_{2^n}^\ast$ as $\widetilde{\mathcal{E}}$ is an ordinary elliptic curve.
So $\# \widetilde{\mathcal{E}} \geq \min_{v \in \F_{2^n}^\ast} \parens{2^n+1 \pm K_n^\ast(v)} = 2^n + 1 - \max_{v \in \F_{2^n}^\ast} |K_n^\ast(v)|$.
Therefore,
\begin{align*}
     6\mult_{\graph F}(a,b) 
     &\geq 2^n -5 - \max_{v \in \F_{2^n}^\ast}|K_n^\ast(v)| + 3\delta_F(1,b) \\
     &\geq 2^n -5 - \max_{v \in \F_{2^n}^\ast}|K_n(v)-1|
\end{align*}
and eq.~(\ref{eq:inverse-minmult}) immediately follows.
\end{proof}

Theorem~\ref{thm:dist-to-Inverse} follows immediately.

\begin{proof}[Proof of Theorem~\ref{thm:dist-to-Inverse}]
    Using the same proof technique as the proof of Theorem~\ref{thm:plat-dist-bound}, we apply Lemma~\ref{lem:MinimumDistance-By-MinimumExcludeMult} to Theorem~\ref{thm:inverse-minmult}.
\end{proof}

We have obtained lower bounds on the exclude multiplicity of any point in $\F_{2^n}^2 \setminus \graph F$, with our lower bound being approximately $\frac{2^n-2^{\frac{n}{2}+1}}{6}$.
Moreover, we provide explicit computations of $e_{\min}(\graph F)$ for $3 \leq n \leq 15$ (with $n$ odd) compared to our lower bound in Table~\ref{table:inverse-mults}.
\begin{table}[ht!]
    \centering
    \begin{tabular}{c|c|c}
         $n$ & $e_{\min}(\graph F)$ 
         & $\ceil{\frac{1}{6} \parens{2^n-5 - \max_{v \in \F_{2^n}^\ast} |K_n(v)-1|}}$  \\
         \hline 
         3 &1 & 0\\
         5 & 3 & 3\\
         7 & 18 & 17 \\
         9 & 77 & 77 \\
         11 & 326  & 326 \\ 
         13 & 1335& 1335\\
         15 & 5401& 5401 
    \end{tabular}
    \caption{The minimum exclude multiplicities of points in $\F_{2^n}^2 \setminus \graph{F}$ with respect to $\graph F$, compared to the lower bound of Theorem~\ref{thm:inverse-minmult} for odd $3 \leq n \leq 15$.}
    \label{table:inverse-mults}
\end{table}

Note that in Table~\ref{table:inverse-mults}, our lower bound agrees with $e_{\min}(\graph F)$ for all odd $5 \leq n \leq 15$ except when $n =7$.
For $n=7$, this happens because, whenever $b \in \F_{2^7} \setminus \F_2$ is such that the associated elliptic curve $\mathcal{E}$ from the proof of Theorem~\ref{thm:inverse-minmult} has the minimal possible number of rational points, then we have $\delta_F(1,b)=2$.
In fact, this occurs for exactly $7$ elements $b \in \F_{2^7}$.

\section{On the existence of linear structures of $\gamma_F$}\label{sec:gammaF-linearstr}

For $\epsilon \in \F_2$, we say that an \textit{$\epsilon$-valued linear structure} of a Boolean function $f \colon \F_2^n \to \F_2$ is a nonzero point $a \in \F_2^n$ such that $D_a f(x)=f(x)+f(x+a)$ takes constant value $\epsilon$.
Since $\gamma_F$ has weight $\binom{2^n}{2}$ when $F$ is APN (see \cite{carletCharpinZinovievCodesBentDES}), it is not balanced, and therefore cannot have a $1$-valued linear structure.
However, not very much is known about whether or not $\gamma_F$ can have $0$-valued linear structures.

In \cite{GammaCarlet}, two classes of APN functions were given such that $\gamma_F$ does not admit nontrivial linear structures.
The first is when $F$ is an APN power function, and the second is when $F$ is such that there exists some $x \in \F_2^n$ with the property that the size of $\set{(a,b) \in (\F_2^n)^2 : D_a D_b F(x) =w}$ does not depend on $w \in \F_2^n \setminus \set{0}$.

We prove that if $F$ is a plateaued APN function with $n$ even, then $\gamma_F$ does not admit nontrivial linear structures, giving a third known class of APN functions for which $\gamma_F$ has no nontrivial linear structures.
Let us rephrase the problem of $\gamma_F$ having linear structures in terms of exclude multiplicity when $F$ is a plateaued APN function.
\begin{lemma}\label{lem:gammaF-nontrivial-linear-structure-iff-max-mult}
Assume $n\geq 4$ is even.
    Let $F \colon \F_2^n \to \F_2^n$ be a plateaued APN function.
    Then $\gamma_F$ has a nontrivial linear structure if and only if there exists $(a,b)\in (\F_2^n)^2 \setminus \graph F$ such that $\mult_{\graph F}(a,b)=\lfloor \frac{2^n}{3}\rfloor = \frac{2^n-1}{3}$.
    Equivalently, $\gamma_F$ has a nontrivial linear structure if and only if there are $2^n$ points in $(\F_2^n)^2 \setminus \graph F$ of the maximum possible exclude multiplicity with respect to $\graph F$.
\end{lemma}

\begin{proof}
    Let $(a,b) \in (\F_2^n)^2 \setminus \graph F$.
    By Proposition~\ref{prop:mult-autocorrelation}, we know that 
    \[
    \mult_{\graph F}(a,b)  = \frac{2^n-2}{6} + \frac{2^{2n}-2\wt(D_{(0,b+F(a))}\gamma_F)}{6 \cdot 2^n}
    \]
    Therefore $\mult_{\graph F}(a,b) = \frac{2^n-1}{3}$ if and only if $\wt(D_{(0,b+F(a))}\gamma_F) = 0$.
\end{proof}

The following result is then an immediate corollary of the fact that plateaued APN functions cannot have an algebraic degree equal to $n$, see \cite{budaghyanCarletHellesetUpperBoundsDegree}.

\begin{proposition}\label{prop:plateaued-nolinearstr}
    Assume $n \geq 4$ is even.
    Let $F \colon \F_2^n \to \F_2^n$ be a plateaued APN function.
    Then $\gamma_F$ has no nontrivial linear structures.
\end{proposition}
\begin{proof}
     Assume that $F$ is plateaued, and by way of contradiction assume that $(0,F(0)+b)$ has an exclude multiplicity of $\frac{2^n-1}{3}$ for some $b \neq 0$.
    Then $\F_2^n\setminus \{0\}$ decomposes into $\frac{2^n-1}{3}$ disjoint triples $\{x,y,z\}$ with  
    \[
    (x + y + z, F(x) + F(y)+F(z)) = (0, b+F(0)).
    \]
    Summing over all triples gives $\sum_{x \in \F_2^n \setminus \set{0}} (x, F(x)) = (0,b + F(0)).$
    Moreover, since $F$ has algebraic degree strictly less than $n$, we know $\sum_{x \in \F_2^n} F(x) = 0$, implying $\sum_{x \in \F_2^n \setminus \set{0}} (x,F(x)) = (0,F(0))$.
    Therefore, $b = 0$, a contradiction.
    Thus, no exclude point of $\graph F$ can have multiplicity equal to $\frac{2^n-1}{3}$ by Proposition~\ref{prop:plateaued-uniformity}, and the result follows from Lemma~\ref{lem:gammaF-nontrivial-linear-structure-iff-max-mult}.
\end{proof}

As shown above, there is an equivalence of $\gamma_F$ having a nontrivial linear structure and $\graph F$ having an exclude point of multiplicity $\frac{2^n-1}{3}$ when $F$ is plateaued APN, and we showed that this is an impossibility.
In the case of $F$ also being quadratic APN, we can then demonstrate that the ortho-derivative of $F$ cannot have trivial component functions.
\begin{corollary}\label{cor:quadratic-linearstructure-orthoderivative}
    Suppose $n \geq 4$ is even, and let $F \colon \F_2^n \to \F_2^n$ be a quadratic APN function.
   Then $\pi_F$ has no components of weight $0$, or equivalently, the image of $\pi_F$ is not contained in a linear hyperplane.
\end{corollary}
\begin{proof}
    Apply Lemma~\ref{lem:gammaF-nontrivial-linear-structure-iff-max-mult} and Proposition~\ref{prop:plateaued-nolinearstr} to eq.~(\ref{eq:mult-orthoderiv-components}).
\end{proof}
We now apply the inequality $\wt(b \cdot \pi_F) \geq \NL(\pi_F)$ for $b \neq 0$ to eq.~(\ref{eq:mult-orthoderiv-components}).
In particular, we have
\begin{equation}\label{eq:mult-nonlin-inequality}
    \mult_{\graph F}(a,b) \leq \frac{2^n-1-\NL(\pi_F)}{3}
\end{equation}
when $F$ is quadratic APN and $(a,b) \notin \graph F$.
By combining eq.~(\ref{eq:mult-nonlin-inequality}) and Proposition~\ref{prop:mult-autocorrelation}, we have that for all even $n$ and any quadratic APN function $F \colon \F_2^n \to \F_2^n$ with $F(0)=0$, and $b \neq 0$,
\[
\Delta_{\gamma_F}(0, b) \leq 2^{2n}-2^{n+1} \NL(\pi_F),
\]
or equivalently,
\[
\wt(D_{(0, b)} \gamma_F) \geq 2^n \NL(\pi_F).
\]
Also, by Theorem~\ref{thm:plateaued-3to1-mults}, one can easily deduce the exact weights of $D_{(0,b)} \gamma_F$ for all $b\neq 0$ when $F$ is a plateaued $3$-to-$1$ function.
Moreover, we have the following inequalities on the weights of the derivatives of $\gamma_F$ of the form $D_{(0,b)}\gamma_F$.

\begin{proposition}
    Suppose $n\geq 4$ is even. 
    Let $F \colon \F_2^n \to \F_2^n$ be a plateaued APN function with $F(0)=0$.
    Then for any $b \neq 0$, the weight of $D_{(0,b)}\gamma_F$ is divisible by $6 \cdot 2^n$ and satisfies
    \[
        0 \leq \wt(D_{(0,b)}\gamma_F) \leq 2^n(2^n - 3 \cdot 2^{\frac{n}{2}-1}+2).
    \]
\end{proposition}
\begin{proof}
    By Proposition~\ref{prop:plateaued-odd-mults}, we know $\mult_{\graph F}(0,b) = 2k+1$ for some non-negative integer $k \geq 0$.
    Therefore, by Proposition~\ref{prop:mult-autocorrelation}, we have $6 \cdot 2^n(2k+1) = 2^n(2^n-2) + 2^{2n}-2\wt(D_{(0,b)}\gamma_F)$.
    So $\wt(D_{(0,b)}\gamma_F) = 2^n(2^n - 6k-4)$.
   Since $n$ is even, $2^n-4$ is divisible by $6$, and so we have $\wt(D_{(0,b)}\gamma_F) \equiv 0 \mod 6 \cdot 2^n$.

    By Theorem~\ref{thm:plat-exclude-bound}, we have  $\mult_{\graph F}(a,b) \geq 2^{\frac{n}{2}-1}-1$ for all $(a,b) \notin \graph F$.
    Moreover, we have $\mult_{\graph F}(a,b) \leq \frac{2^n-1}{3}$ by \cite[Corollary 3.4]{quadspaper}.
    So, for $b \neq 0$, we have $2^{\frac{n}{2}-1}-1 \leq \mult_{\graph F}(0,b) \leq \frac{2^n-1}{3}$, and the result follows.
\end{proof}

\section{Open problems}\label{sec:open-ques}

To conclude the paper, we list a few open problems relating to the exclude multiplicities of the graphs of APN functions.

\begin{openp}
    Establish a lower bound on the exclude multiplicities of the graph of the Dobbertin function or all APN power functions in general.
\end{openp}

\begin{openp}\label{openprob2}
    Establish a better lower bound on the exclude multiplicities of the graph of a plateaued (or quadratic) APN function or find an infinite family that tightly attains our bound from Theorem~\ref{thm:plat-exclude-bound}.
\end{openp}

\begin{openp}
    Study the nonlinearity of the ortho-derivative $\pi_F$ of a quadratic APN function $F$.
    Is it possible for $\pi_F$ to have zero nonlinearity? 
\end{openp}

\begin{openp}
    For $n$ even, find an example of a non-plateaued APN function $F \colon \F_2^n\to \F_2^n$ whose graph has an exclude point of even multiplicity, or prove such a function does not exist. 
\end{openp}

Currently, it is unclear whether any of the known lower bounds on the distances between APN functions can be attained by pairs of APN functions via existing constructions or secondary constructions (see for instance, \cite{Budaghyan009150} or \cite{ArshadChanging}).
The following open problem seems like an interesting area of future research.

\begin{openp}
    In the same spirit as Open Problem~\ref{openprob2}, it would be interesting to determine whether or not any of the lower bounds in Theorems~\ref{thm:plat-dist-bound}, \ref{thm:distance-to-plat3to1}, or ~\ref{thm:dist-to-Inverse} are tight.
    In particular, more research on the exact distances between pairs of known APN functions (including those arising from secondary constructions) would be valuable.
\end{openp}

\section*{Declarations}

\begin{description}
    \item[Ethics approval and consent to participate:] Not applicable. 
    \item[Consent for publication:] Not applicable. 
    \item[Availability of data and materials:] This study does not generate or analyze datasets.
    \item[Competing interests:] The authors declare they have no competing interests.
    \item[Funding:] This research received no external funding.
    \item[Authors' contributions:] All authors contributed equally to this work.
    \item[Acknowledgments:] We thank Claude Carlet and Larry Rolen for useful discussions and helpful feedback that improved earlier drafts of this paper. We also thank the anonymous reviewers who helped improve previous versions of this paper.
\end{description}

\printbibliography

\end{document}